\documentclass[11pt,reqno]{amsart}
\usepackage[dvips]{graphicx,color}
\usepackage{amssymb,amscd,latexsym,epsfig, xypic}
\usepackage[mathscr]{euscript}

\DeclareFontFamily{OT1}{rsfs}{}
\DeclareFontShape{OT1}{rsfs}{n}{it}{<-> rsfs10}{}
\DeclareMathAlphabet{\curly}{OT1}{rsfs}{n}{it}

\renewcommand\O{\mathcal O}
\newcommand\I{\mathscr I}

\newcommand\PP{\mathbb P}
\newcommand\C{\mathbb C}
\newcommand\Q{\mathbb Q}
\newcommand\R{\mathbb R}
\newcommand\Z{\mathbb Z}

\newcommand\Into{\ar@{^{ (}->}[r]}

\newfont{\bigtimesfont}{cmsy10 scaled \magstep5}
\newcommand{\bigtimes}{\mathop{\lower0.9ex\hbox{\bigtimesfont\symbol2}}}

\newcommand\Hom{\operatorname{Hom}}

\newcommand\Ext{\operatorname{Ext}}

\newcommand\ext{\operatorname{ext}}
\newcommand\Aut{\operatorname{Aut}}

\newcommand\Spec{\operatorname{Spec}\,}
\newcommand\Hilb{\operatorname{Hilb}}

\newcommand\Gr{\operatorname{Gr}}

\newcommand{\len}{\operatorname{len}}

\newcommand\bra{\langle}
\newcommand\ket{\rangle}

\newcommand{\dt}{\operatorname{DT}}
\newcommand{\pii}{\operatorname{PI}}
\newcommand{\p}{\operatorname{p}}

\newcommand{\q}{\operatorname{q}}
\newcommand{\n}{\operatorname{n}}
\newcommand{\ii}{\operatorname{i}}
\newcommand{\s}{\operatorname{S}}
\newcommand{\U}{\operatorname{U}}

\newcommand\beq[1]{\begin{equation}\label{#1}}
\newcommand\eeq{\end{equation}}
\newcommand\beqa{\begin{eqnarray*}}
\newcommand\eeqa{\end{eqnarray*}}

\makeatletter \@addtoreset{equation}{section} \makeatother

\newtheorem{lem}[equation]{Lemma}

\newenvironment{rmk}{\noindent\textbf{Remark}.}{\\}

\newenvironment{exa}{\noindent\textbf{Example}.}{\\}

\title[]{D0-D6 states counting and GW invariants}
\author{Jacopo Stoppa}
\date{14 December 2009}
\address{Department of Pure Mathematics and Mathematical Statistics, University of Cambridge, Wilberforce Road, Cambridge CB3 0WB, UK}
\email{J.Stoppa@dpmms.cam.ac.uk}
\begin{document}  
\begin{abstract} We describe a correspondence between the virtual number of torsion free sheaves locally free in codimension 3 on a Calabi-Yau 3-fold and the Gromov-Witten invariants counting rational curves in a family of orbifold blowups of the weighted projective plane $\PP(-\operatorname{ch}_3, \operatorname{ch}_0, 1)$ (with a tangency condition of order gcd$(-\operatorname{ch}_3, \operatorname{ch}_0)$). This result is a variation of the GW/quiver representations correspondence found by Gross-Pandharipande, when one changes the centres and orders of the blowups. We build on a small part of the theories developed by Joyce-Song and Kontsevich-Soibelman for wall-crossing formulae and by Gross-Siebert-Pandharipande for factorisations in the tropical vertex group.  
\end{abstract}
\maketitle
\section{Introduction}
\subsection{A D0-D6/GW correspondence}
Let $X$ be a projective Calabi-Yau threefold with $H^1(\O_X) = 0$ and topological Euler characteristic $\chi$. In this paper we are concerned with torsion free coherent sheaves of $\O_X-$modules which are isomorphic to the trivial vector bundle of some rank outside a finite length subscheme, `D0-D6 states'. We write the Chern character as
\[(a, r):= (r, 0, 0, -a) \in \bigoplus^{3}_{i = 0} H^{2i}(X, \Z)\]
where $a = -\operatorname{ch}_3$, $r = \operatorname{ch}_0$. The key feature of these sheaves for us is that they can be `counted' in a suitable way.

For rank $r = 1$ these are the \emph{ideal sheaves} of $0$-dimensional subschemes of $X$, and $a$ is the length of the subscheme. They are Gieseker stable with respect to any ample line bundle $\O_X(1)$ and have a fine moduli space $\mathcal{M}(a, 1)\cong\Hilb^a(X)$ with a symmetric obstruction theory in the sense of \cite{bf}.  Donaldson-Thomas theory \cite{richard} produces \emph{integer} virtual counts $\#^{vir}\Hilb^a(X)$. The $0$-dimensional Donaldson-Thomas partition function $\sum_{a\geq 1}\#^{vir}\Hilb^a(X)t^a$ has been computed as $M(-t)^{\chi}$ in \cite{bf}, \cite{lp}, \cite{li} (here and in the rest of the paper $M(t)$ is the MacMahon function, the generating series for 3-dimensional partitions).

For $r = 0$ we are looking instead at direct sums of \emph{structure sheaves} of $0$-dimensional subschemes. Their Chern character is $(-a, 0)$ where $a$ is the length of the $\O_X$-module. Because of automorphisms their moduli space is an Artin stack $\mathcal{M}(-a, 0)$ and it was not clear how to count them correctly until recently. However as a very special case of the fundational work of Joyce-Song \cite{joy} we now have generalised Donaldson-Thomas invariants $\bar\dt(-a, 0)\in \mathbb{Q}$; they are \emph{not} in general the weighted Euler characteristic of the stack $\mathcal{M}(-a, 0)$ with respect to its canonical Behrend function. It is shown in \cite{joy} Section 6.3 that $\bar\dt(-a, 0) = -\chi \sum_{m\mid a}\frac{1}{m^2}$.\\ 

Now let $\mathcal{A}$ be the abelian category of coherent sheaves on $X$ which are locally free in codimension $3$. Its numerical Grothendieck group $K(\mathcal{A})$ is isomorphic to $\Z^2$ spanned by the classes $\mu = [\O_x], \gamma = [\O_X]$ (where $x \in X$ is any closed point). Consider the homomorphism $Z\!: K(\mathcal{A})\to \C$ given by $Z(\mu) = -1, Z(\gamma) = i$. Since $Z$ maps the effective cone $K^{+}(\mathcal{A})$ into $\{\rho \exp(i\varphi) : \rho > 0, 0 < \varphi \leq \pi\}\subset \C$ it determines a \emph{stability condition} on $\mathcal{A}$. The semistable objects $\mathcal{A}^{ss} \subset \mathcal{A}$ are in fact the torsion free sheaves which are isomorphic to the trivial vector bundle of some rank in codimension $3$ (see \cite{ks} Section 6.5). There is an Artin stack of objects of $\mathcal{A}$ which as in \cite{joy} Section 5.1 is locally 2-isomorphic to the zero locus of the gradient of a regular function on a smooth scheme, and $Z$ gives an admissible stability condition in the sense of \cite{joy} Section 3.2.

Joyce-Song theory then yields invariants $\bar\dt(a, r) \in \mathbb{Q}$ which count $Z$-semista-\\ble objects in a suitable way. We also refer to the very recent paper of Toda \cite{toda} for a number of foundational results on higher rank DT invariants in the sense of this paper. When $(a, r)$ is a primitive class the $\bar\dt$ coincide with the DT invariants of \cite{richard}. In particular we recover the numbers counting ideal sheaves and $0$-dimensional subschemes. Notice also that one can show directly that $\bar\dt(0, r) = \frac{1}{r^2}$ and $\bar\dt(a, r) = 0$ for $a = 1, \dots, r-1$ (see \cite{joy} Example 6.1 and \cite{ks} Section 6.5). Therefore in the rest of this paper we concentrate on $\bar\dt(a, r)$ with $a \geq r$. 

It is sometimes possible to compute higher rank D0-D6 numbers more or less directly, using Behrend functions. The reader can find an explicit calculation of $\bar\dt(2, 2) = -\frac{5}{4}\chi$, together with a brief introduction to Joyce-Song invariants in this context in the appendix.\\

A rather different take on the numbers $\bar\dt(a, r)$ is motivated by the work of Kontsevich-Soibelman \cite{ks} and Gross-Siebert-Pandhari\-pande \cite{gps}. The example of D0-D6 states is studied in particular in \cite{ks} Section 6.5.  According to their general theory, Kontsevich-Soibelman conjecture that one can extract \emph{integers} $\Omega$ from the $\bar\dt$ (their `BPS invariants') by inverting the relation $\bar\dt(a, r) = \sum_{m \geq 1, m \mid (a, r)}\frac{1}{m^2}\Omega(\frac{a}{m}, \frac{r}{m})$. Moreover they conjecture that these BPS numbers should be completely determined by a simple identity taking place in the `tropical vertex group', a Lie group of formal symplectomorphisms of the 2-dimensional algebraic torus.

We will argue that this identity should be seen as a commutator expansion in the tropical vertex group. Gross-Pandharipande-Siebert \cite{gps} have developed a theory which interprets such commutators in the tropical vertex group in terms of genus zero Gromov-Witten invariants with a tangency condition. We will explain how `counting' (in the sense of Joyce-Song) the torsion free sheaves on $X$ which are isomorphic to the trivial vector bundle of some rank in codimension 3 becomes equivalent to computing the genus zero Gromov-Witten invariants (with a tangency condition) of some explicit 2-dimensional orbifolds, depending only on $\chi$ and the given $K$-theory class. 

In this paper we prove the Kontsevich-Soibelman identity for $\Omega(a, r)$ for rank $r \leq 3$, use it to deduce the integrality of the relevant BPS numbers and outline an argument for the KS identity (but not integrality) for arbitrary $r$ (for a different situation in which one can show that the Joyce-Song invariants satisfy the relevant KS equation see \cite{dia}). Very recently Toda also studied these $r = 2$ DT invariants, in particular the partition function is computed in \cite{toda} Theorem 1.2 and integrality of BPS states is proved in loc. cit. Theorem 1.3. 

Finally we explain the connection with GW invariants through the results of Gross-Siebert-Pandharipande, which can be expressed as follows.\\

\emph{The BPS numbers $\Omega(h a, h r)$ counting torsion free sheaves with K-theory class a multiple of the primitive class given by coprime $\operatorname{ch}_0=r$ and $\operatorname{ch}_3 = -a$, locally free in codimension 3 on a CY 3-fold with Euler characteristic $\chi$, satisfy the identity in the ring of formal power series $\C[[x,y]]$
\begin{equation}\label{GW}
\prod_{h \geq 1} \exp\left(\sum_{|P_{\chi}| = h a} h N[P_{\chi}](-1)^{h( a +  r)}x^{h a}y^{h r}\right) = \prod_{h \geq 1}(1 - (-1)^{h^2 a r}x^{h a} y^{h r})^{\Omega(h a, h r)}
\end{equation} 
where $N[P_{\chi}]$ are the Gromov-Witten invariants of a family of orbifold blowups of the toric surface given be the fan $\{(-1,0), (0, -1), (a, r)\}\subset \R^2$ (with some points removed and a tangency condition of order $h$ along a smooth divisor), parametrised by graded ordered partitions $P_{\chi}$ (depending on the Euler characteristic $\chi$) with size $|P_{\chi}| = h a$}.\\

The base toric surface is the weighted projective place $\PP(a, r, 1)$ with some points removed. The index $h = \operatorname{gcd}(-\frac{1}{2}\operatorname{ch}_3, \operatorname{ch}_0)$ for sheaves corresponds to the order of tangency for holomorphic curves along the divisor $D_{out}$ dual to $(a, r)$. 

This result should be compared with the correspondence described by Gross-Pandharipande in \cite{gp}, Corollary 3, building on \cite{gps} and the work of Reineke \cite{rein}. In essence Gross-Pandharipande show that the Euler characteristics of the moduli spaces for stable representations of the $m$-Kronecker quiver can be computed in terms of GW theory. The above correspondence says that, in a different region of the tropical vertex group, quiver representations are replaced by D0-D6 states. Geometrically, starting with the same base orbifold $\PP(a, r, 1)$, the invariants for representations of the $m$-Kronecker quiver (with dimension vector proportional to $(a, r)$) are recovered for ordinary blowups along the divisors $D_1, D_2$ dual to $(-1,0), (0, -1)$, parametrised by suitable partitions $P_1, P_2$ of length proportional to $m$. For D0-D6 states we blow up only once along $D_2$, and we also blow up $D_1$ along a partition; but the crucial difference is that now there are also `higher order' corrections, or more precisely \emph{orbifold} blowups in addition to ordinary ones, and in turn these are parametrised by graded ordered partitions of length proportional to $\chi$. The precise statement will be given in section \ref{GW_sect}.
\subsection{Comparison with a physics result} Even before their rigorous definition by Joyce-Song, Cirafici-Sinkovics-Szabo \cite{szabo} have addressed the problem of computing the punctual invariants $\dt(a, r)$ supported at the origin of the affine Calabi-Yau $\C^3$. However as they explain in ibid. Section 7.2 their physical approach based on noncommutative deformation and localisation is only valid in the regime called Coulomb phase, where they compute the partition function simply as $M((-1)^r t)^r$. As we will see this is very different from the result one would get by setting $\chi = 1$ in the correspondence \eqref{GW}. We will argue that the above result in the Coulomb phase (i.e. in physics terminology, in the limit when the gauge group $U(r)$ breaks down to $U(1)^r$) can be seen as a limit of our result when the central charge $Z$ becomes degenerate. In this case we find the partition function $\frac{1}{r^2}M((-1)^r t)^r$, and it seems natural to say its BPS is $M((-1)^r t)^r$. Since $Z$ is degenerate the theory of \cite{gps} cannot be applied and it seems that the Coulomb phase cannot be seen by holomorphic curves.
\subsection{Identity in the tropical vertex group} We follow the notation of \cite{gps}. The tropical vertex group $G$ is a closed subgroup of $\Aut_{\C[[t]]}(\C[x,x^{-1},y,y^{-1}][[t]])$ in the $(t)$-adic topology. It is the $(t)$-adic completion of the subgroup generated by the automorphisms of the form
\begin{equation*}
\theta_{(a,r),f}(x) = f^{-r}\cdot x,\,\,\,\theta_{(a,r),f}(y) = f^a\cdot y
\end{equation*}
with $(a, r) \in \Z^2$ and $f$ a formal power series in $t$ of the form
\begin{equation*}
f = 1 + t x^a y^r \cdot g(x^a y^r, t), \,\,\,g(z, t)\in\C[z][[t]].
\end{equation*}
Alternatively one can see $G$ as a subgroup of the group of formal $1$-parameter families of automorphisms of the algebraic $2$-torus $\C^*\times\C^*$; by direct computation $G$ preserves the standard holomorphic symplectic form $\frac{dx}{x}\wedge\frac{dy}{y}$. 

A basic feature of $G$ is that two elements $\theta_{(a, r), f}, \theta_{(a',r'),f'}$ with $(a',r')$ a multiple of $(a, r)$ commute.

The group $G$ contains some special elements
\begin{equation*}
T_{a, r} = \theta_{(a, r), 1 - (-1)^{a r}(t x)^a (t y)^r},
\end{equation*}
and for $\Omega \in \mathbb{Q}$ we define 
\begin{equation*}
T^{\Omega}_{a, r} (x, y) = \theta_{(a, r), (1 - (-1)^{a r}(t x)^a (t y)^r)^{\Omega}}.
\end{equation*} 
The notation makes sense from the point of view of Lie groups, since
\begin{equation*}
T_{a, b} = \theta_{(a, r), f} = \exp(\log(f)\partial) 
\end{equation*}
for $f = 1 - (-1)^{a r}(t x)^a (t y)^r$ and some $\partial \in \Z dx \oplus \Z dy$ so $T^{\Omega}_{a, r}$ corresponds to $\exp(\Omega\log(f)\partial) = \exp(\log(f^{\Omega})\partial) = \theta_{(a, r), f^{\Omega}}$ (see \cite{gps} section 1.1 and \cite{huy} for the general setup). Notice that in particular
\begin{equation*}
(T_{a,r})^{-1} = T^{-1}_{a,r} = \theta_{(a, r), (1-(-1)^{a r}(t x)^a (t y)^r)^{-1}}.
\end{equation*}

By a fundamental result of Kontsevich-Soibelman every automorphism in the tropical vertex group has a unique ordered product expansion  
\begin{equation}
g = \prod^{\to}_{(a, r)\in\Z^{2}_+}T^{\Omega(a,r)}_{a, r}
\end{equation}
where $\Z^2_+ \subset \Z^2$ means $\{a, r \geq 0\}\setminus\{0\}$ (for a precise definition of the ordered product $\prod^{\to}$ see \cite{ks} Section 2.2 and for the proof of an equivalent statement see e.g. \cite{gps} Theorem 1.4).

Let us now go back to the category $\mathcal{A}$. According to Kontsevich-Soibelman and Joyce-Song one introduces BPS invariants associated to the $\bar\dt$ as 
\begin{equation*}
\Omega(a,r) = \sum_{m \geq 1, m\mid (a, r)}\frac{\mu(m)}{m^2}\bar\dt\left(\frac{a}{m},\frac{r}{m}\right),
\end{equation*} 
where $\mu(m)$ is the M\"obius function (with $\mu(1) = 1, \mu(2) = -1, \mu(3) = -1$, ...). 

For rank one we get simply $\Omega(a, 1) = \bar\dt(a,1)$ for $a\geq 1$ since these classes are primitive. On the other hand one can compute $\Omega(-a, 0) = -\chi$ for $a \geq 1$, see \cite{joy} Section 6.3. Another example is $\Omega(2,2) = -\chi$ and can be found in the appendix.\\
\begin{rmk} Computation suggests the identity
\begin{equation*}
\Omega(a, a-i) = \Omega(a, i) \text{ for } i = 1, \dots, a-1  
\end{equation*}
and the identity
\begin{equation*}
\Omega(a, a) = -\chi \text{ for } a \geq 1.
\end{equation*}
Both could have an interesting interpretation in terms of rational curves under the D0-D6/GW correspondence \eqref{GW}.
\end{rmk}

In \cite{ks} Section 6.5 Kontsevich-Soibelman write down an identity in the tropical vertex group which should be satisfied by the BPS invariants $\Omega$, namely
\begin{equation}\label{wallcrKS}
\prod_{a \geq 1} T^{-\chi}_{a, 0}\cdot T_{0, 1} = \prod^{\to}_{a \geq 0, r \geq 1} T^{\Omega(a, r)}_{a, r}\cdot \prod_{a \geq 1} T^{-\chi}_{a, 0}. 
\end{equation}
According to the factorisation theorem recalled above this formula would determine the $\Omega(a, r)$ uniquely.\\

Let us explain the origin of the formula \ref{wallcrKS}, referring to loc. cit. for a detailed discussion. In \cite{ks} Kontsevich-Soibelman propose an alternative approach to generalised Donaldson-Thomas invariants counting semistable objects is suitable triangulated categories with respect to a Bridgeland stability condition. In particular they propose universal formulae for how the BPS invariants change as the stability condition moves in the space $Stab$. Locally these remain constant, but there are \emph{walls} in $Stab$ on crossing which the $\Omega$ change according to formulae of the form of \ref{wallcrKS}. This theory is still conjectural in parts. However we can apply it formally to $\mathcal{A}$. For this embed $\mathcal{A}$ as the heart of a $t$-structure in the triangulated category $\mathcal{D}$ generated by extensions from $\O_X$ and $\O_{x}$ for $x \in X$. The slope function $Z$ becomes a central charge on $\mathcal{D}$ defining a Bridgeland stability condition, and the Joyce-Song BPS invariants $\Omega$ conjecturally coincide with the Kontsevich-Soibelman invariants counting $Z$-semistable objects with phases in some fixed sector. Kontsevich-Soibelman deform $Z$ by prescribing $Z_{\tau}([\O_X]) = i, Z_{\tau}([\O_x]) = \exp(i\tau\pi)$ for $\tau\in[1,\frac{3}{2}\pi + \epsilon)$ for sufficiently small $\epsilon > 0$. As soon as $t > 1$ the heart $\mathcal{A}$ jumps to its tilting $\mathcal{A}'$ with respect to $0$-dimensional sheaves (i.e. $-[\O_x]$ becomes effective in $K(\mathcal{A}')$ as it is the class of $\O_x[-1]$). The BPS invariants however remain unchanged until $\tau \geq \frac{3}{2}\pi$. For $\tau > \frac{3}{2}\pi$ the only semistable objects have unmixed classes which are multiples of either $[\O_X]$ or $-[\O_x]$. The latter objects have BPS invariants $-\chi$ as we already discussed. By assumption $\O_X$ is rigid and so according to \cite{joy} Section 6.1 we have $\Omega(r\O_X)= \delta_{r,1}$. The equality of factorisations \eqref{wallcrKS} then becomes the Kontseich-Soibelman wall-crossing formula for this situation. 
\subsection{The tropical vertex for GW invariants} Clearly the wall-crossing formula \eqref{wallcrKS} can be rewritten as an expansion for a commutator in $G$,
\begin{equation*}
(T_{0,1})^{-1}\cdot\left(\prod_{a \geq 1} T^{-\chi}_{a, 0}\right)\cdot T_{0, 1}\left(\prod_{a \geq 1} T^{-\chi}_{a, 0}\right)^{-1} = \prod^{\to}_{a \geq 1, r \geq 1} T^{\Omega(a, r)}_{a, r}. 
\end{equation*} 
Using the definition of $T_{a,0}$ and the well known product formula for the McMahon function $M(x) = \prod_{a\geq 1}(1 - x^a)^{-a}$ one can check that this is equivalent to 
\begin{equation*}
(\theta_{0,1})^{-1}\cdot\theta_{(1,0), M(-tx)^{\chi}}\cdot\theta_{0,1}\cdot \theta_{(1,0), M(-tx)^{-\chi}} = \prod^{\to}_{a\geq 1, r\geq 1, \operatorname{gcd}(a, r)=1} \theta_{(a, r), f_{a, r}}
\end{equation*}
with
\begin{equation*}
f_{(a, r)} = \prod_{k \geq 1}(1 - (-1)^{k^2 a r}(t x)^{k a} (t y)^{k r})^{\Omega(ka, kr)}.
\end{equation*}
This is precisely the kind of commutator expansions studied in \cite{gps} by Gross-Pandharipande-Siebert. They have shown that for the ordered product factorisation 
\[[\theta_{(a', r'),f'},\theta_{(a'',r''),f''}] = \prod^{\to}_{a\geq 1, r\geq 1, \operatorname{gcd}(a, r)=1} \theta_{(a, r), f_{a, r}}\]
of the commutator of two generators of $G$ one can write the coefficients of the power series $\log f_{a,r}$ in terms of the GW invariants of orbifold blowups of a toric surface $X_{a,r}$ (with a tangency condition). We will describe explicitely how this result applies in our case, relating D0-D6 states to GW invariants of orbifolds.
\subsection{Plan of the paper} In sections \ref{KS_sect} and \ref{JS_sect} we show that the Joyce-Song invariants satisfy the relevant Kontsevich-Soibelman identities for rank up to $3$, and we use these identities to prove the integrality of the BPS invariants which arise. At the end of section \ref{JS_sect} we outline an argument for the KS identities (but not integrality) for all ranks. Finally in section \ref{GW_sect} we briefly review the theory of Gross-Pandharipande-Siebert and apply it to our special case, thus obtaining the required formulae for D0-D6 states counts in terms of GW invariants of orbifold blowups of weighted projective planes.\\

\noindent\textbf{Acknowledgements.} I wish to thank the participants in the study group `Behrend functions and DT invariants', Max Planck Institute for Mathematics, Bonn, April-May 2009 and in the seminars SFB/TR 45 Bonn-Essen-Mainz `Motivic Donaldson-Thomas invariants', May-July 2009. In particular I was motivated by several conversations with Daniel Huybrechts while setting up the latter seminar. I am also grateful to Richard Thomas and Rahul Pandharipande for important suggestions and comments.
\section{Kontsevich-Soibelman side}\label{KS_sect}
\subsection{Baker-Campbell-Hausdorff formula} Write $\Gamma$ for the lattice $\Z^2$ with basis $\gamma = (0, 1), \mu = (1, 0)$ and anti-symmetric bilinear form $\bra\gamma, \mu\ket = -1$ (in the categorical picture above this corresponds to $\gamma = [\O_X], \mu = [\O_x[-1]]$). The positive cone $\Gamma_{+} \subset \Gamma$ is given by those elements with nonnegative components $(a, r), a + r \geq 1$.

Consider the $\Gamma_{+}$-graded Lie algebra $\mathfrak{g}$ generated over $\C$ by symbols $e_{\eta}, \eta \in \Gamma_{+}$ with bracket
\begin{equation}\label{lie}
[e_{\xi}, e_{\eta}] = (-1)^{\bra\xi, \eta\ket}\bra\xi, \eta\ket e_{\xi + \eta}.
\end{equation}
Then writing $\eta = (a, r)$ for an element of $\Gamma$ there is a natural identification
\begin{equation}\label{ops}
T_{(a, r)} = T_{\eta} = \exp\left(-\sum_{n \geq 1}\frac{e_{n \eta}}{n^2}\right)
\end{equation}
seeing the automorphism $T_{\eta}$ as an element of the exponential of the completion of $\mathfrak{g}$ (see \cite{ks} Section 1.4 for this identification, and notice that here we are replacing the $t$-grading with the finer $\Gamma_{+}$-grading). We rewrite the KS formula \eqref{wallcrKS} as
\begin{equation}\label{wallcrKS2} 
\prod_{n \geq 1} T^{-\chi}_{n \mu}\cdot T_{\gamma}\cdot \left(\prod_{n \geq 1} T^{-\chi}_{n \mu}\right)^{-1} = \prod^{\to}_{n \geq 0, r \geq 1} T^{\Omega(n\mu + r\gamma)}_{n\mu + r\gamma}.
\end{equation}
Let us define operators
\[A = \chi \sum_{n \geq 1} \sum_{i \geq 1} \frac{e_{i n\mu}}{i^2},\,\,\, B = -\sum_{j \geq 1}\frac{e_{j\gamma}}{j^2}.\]
In what follows we will denote the left and right hand sides of \eqref{wallcrKS2} simply by $lhs$, $rhs$. Using repeatedly \eqref{lie} and \eqref{ops} the left hand side of \eqref{wallcrKS2} can be rewritten as
\begin{equation*}
lhs = \exp(A) \exp(B) \exp(-A).
\end{equation*} 
We will use the following form of the Baker-Campbell-Hausdorff formula, 
\begin{align}
\nonumber\exp(A) \exp(B) \exp(-A) &= \exp\left(B + \sum_{k\geq 1} \frac{\operatorname{Ad}^k_A(B)}{k!}\right)\\
&= \exp(B + [A, B] + \frac{1}{2}[A, [A, B]] + \dots)
\end{align}
Let us write $\n, \ii$ for multi-indexes of length $k \geq 1$ with integer entries $n_l, i_l \geq 1$, and $\n\cdot\ii = \sum^k_{l = 1} n_l i_l$ for their ordinary scalar product. For $k \geq 1$ we can compute
\begin{equation*}
\operatorname{Ad}^k_A(B) = -\chi^k \sum_{\n, \ii}\sum_{j \geq 1}(-1)^{j\n\cdot\ii}j^{k-2}\frac{\prod n_l}{\prod i_l}e_{\n\cdot\ii \mu + j\gamma}.
\end{equation*}
Thus we find 
\begin{equation}\label{boundary}
\log(lhs) = -\sum_{j \geq 1}\frac{e_{j\gamma}}{j^2}-\sum_{k \geq 1}\frac{\chi^k}{k!}\sum_{\len(\n) = \len(\ii) = k}\sum_{j \geq 1}(-1)^{j\n\cdot\ii}j^{k-2}\frac{\prod n_l}{\prod i_l}e_{\n\cdot\ii \mu + j\gamma}.
\end{equation}
\subsection{Rank $r = 1$} Consider the subspace $\mathfrak{g}_{> 1}$ of $\mathfrak{g}$ generated by $e_{\eta}$ with $\bra\mu,\eta\ket > 1$. By \eqref{lie} this is an ideal $\mathfrak{g}_{> 1} < \mathfrak{g}$, so we can form the quotient Lie algebra $\mathfrak{g}/\mathfrak{g_{> 1}}$. The right hand side $rhs$ of \eqref{wallcrKS2} can be projected via 
\[\pi_{\leq 1}\!:\exp(\mathfrak{g}) \to \exp(\mathfrak{g}/\mathfrak{g}_{> 1})\]
taking the form
\[\pi_{\leq 1}(rhs) = \pi_{\leq 1}({T}_{1, 0}) \prod_{a \geq 1} \pi_{\leq 1}({T}^{\Omega(a, 1)}_{a, 1}).\]
Now 
\[\pi_{\leq 1}(T_{1,0}) = \exp(-e_{\gamma}),\,\,\, \pi_{\leq 1}(T^{\Omega(a, 1)}_{a, 1}) = \exp(-\Omega(a,1) e_{a \mu + \gamma}),\]
and \emph{in the quotient} we have $[e_{a\mu + \gamma}, e_{a'\mu + \gamma}] = 0$, so
\[\log(rhs) = -e_{\gamma} - \sum_{a \geq 1}{\Omega(a, 1)e_{a\mu + \gamma}}.\]
Comparing with the left hand side gives the rank $r = 1$ KS formula
\begin{equation}
\Omega(a, 1) = (-1)^a\sum_{k \geq 1} \frac{\chi^k}{k!} \sum_{\operatorname{len}(\n) = \operatorname{len}(\ii) = k, \n\cdot\ii = a} \frac{\prod n_l}{\prod i_l}.
\end{equation}
These are the usual $0-$dimensional DT invariants, but this particular way to represent them turns out to be very useful for the generalisation to higher rank.
\begin{rmk} The $r = 1$ formula thus gives
\begin{equation*}
\sum_{a \geq 0} t^a (-1)^a\sum_{k \geq 1} \frac{\chi^k}{k!} \sum_{\operatorname{len}(\n) = \operatorname{len}(\ii) = k, \n\cdot\ii = a} \frac{\prod n_l}{\prod i_l} = M(-t)^{\chi}.
\end{equation*}
In general comparing with \eqref{boundary} above we find
\begin{equation}
\sum_{a \geq 0}t^a(-1)^{r a}\sum_{k \geq 1}\frac{\chi^k}{k!}\sum_{\len(\n) = \len(\ii) = k, \n\cdot\ii = a}\sum_{r \geq 1}r^{k-2}\frac{\prod n_l}{\prod i_l}e_{\n\cdot\ii \mu + r\gamma} = \frac{1}{r^2}M((-1)^r t)^{r\chi}.
\end{equation}  
By this computation and according to \cite{ks} Section 2.3 the partition function for rank $r$ BPS states \emph{for the degenerate stability condition on the wall} is 
\begin{equation}
Z^{BPS}_{r, degen}(t) = \frac{1}{r^2}M((-1)^r t)^{r\chi}.
\end{equation}
These are not integers and are not expected to be since the degenerate stability condition is not `generic' in the sense of \cite{joy} Section 1.4. But there is an obvious way to make them integral, namely taking $r^2 Z^{BPS}_{r, degen}(t)$. This agrees with the physics result from \cite{szabo}, i.e. in the Coulomb phase. It would be interesting to understand this correspondence better. 
\end{rmk}
\subsection{Rank $r=2$} Similarly we can work out a formula for $\Omega(a, 2)$. Let us consider the quotient $\mathfrak{g}/\mathfrak{g}_{>2}$ with projection $\pi_{\leq 2}\!: \exp(\mathfrak{g})\to\exp(\mathfrak{g}/\mathfrak{g}_{> 2})$; the projection of the right hand side is
\begin{equation*}
\pi_{\leq 2}(rhs) = \pi_{\leq 2}(T_{0, 1})\prod^{\to}_{a \geq 1, 1 \leq b \leq 2} \pi_{\leq 2}(T^{\Omega(a,b)}_{(a, b)}).
\end{equation*}
Explicitly, 
\begin{align*}
\pi_{\leq 2}(T^{\Omega(0, 1)}_{0,1}) &= \exp(-e_{\gamma}-\frac{1}{4}e_{2\gamma}),\\
\pi_{\leq 2}(T^{\Omega(a, 1)}_{a,1}) &= \exp(-\Omega(a,1)(e_{a\mu +\gamma}+\frac{1}{4}e_{2a\mu+2\gamma})),\\
\pi_{\leq 2}(T^{\Omega(a, 2)}_{a, 2}) &= \exp(-\Omega(a,2)e_{a\mu + 2\gamma}).
\end{align*}
Since by \eqref{lie} $[\mathfrak{g}/\mathfrak{g}_{>2},[\mathfrak{g}/\mathfrak{g}_{>2},\mathfrak{g}/\mathfrak{g}_{>2}]] = 0$ the Baker-Campbell-Hausdorff formula gives 
\begin{align*}
\nonumber &\log(rhs) = -e_{\gamma}-\frac{1}{4}e_{2\gamma} - \sum_{a \geq 1}\Omega(a, 1)(e_{a\mu +\gamma}+\frac{1}{4}e_{2a\mu+2\gamma})\\
&-\sum_{a \geq 1}\Omega(a, 2)(e_{a\mu + 2\gamma})+\frac{1}{2}\sum_{a' < a''}(-1)^{a'-a''}(a'-a'')\Omega(a', 1)\Omega(a'', 1)e_{(a'+a'')\mu + 2\gamma}.
\end{align*}  
When $a$ is odd (i.e. in the primitive case) comparing with $lhs$ gives
\begin{align}
\nonumber\Omega(a, 2) &= \sum_{k \geq 1} 2^{k-2}\frac{\chi^k}{k!} \sum_{\operatorname{len}(\n) = \operatorname{len}(\ii) = k, \n\cdot\ii = a} \frac{\prod n_l}{\prod i_l}\\
&+ \frac{(-1)^a}{2}\sum_{a' < a'', a' + a'' = a}(a'-a'')\Omega(a', 1)\Omega(a'', 1),
\end{align}
while for $a$ even there is an additional term,
\begin{align}
\nonumber\Omega(a, 2) &= \sum_{k \geq 1} 2^{k-2}\frac{\chi^k}{k!} \sum_{\operatorname{len}(\n) = \operatorname{len}(\ii) = k, \n\cdot\ii = a} \frac{\prod n_l}{\prod i_l}\\
\nonumber &+ \frac{(-1)^a}{2}\sum_{a' < a'', a' + a'' = a}(a'-a'')\Omega(a', 1)\Omega(a'', 1)\\
&- \frac{1}{4}\Omega(a/2,1).
\end{align}
According to the definition of BPS invariants in each case we get
\begin{align}
\nonumber\bar\dt(a,2) &= \sum_{k \geq 1} 2^{k-2}\frac{\chi^k}{k!} \sum_{\operatorname{len}(\n) = \operatorname{len}(\ii) = k, \n\cdot\ii = a} \frac{\prod n_l}{\prod i_l}\\
&+ \frac{(-1)^a}{2}\sum_{a' < a'', a'+a''=a}(a'-a'')\Omega(a', 1)\Omega(a'', 1).
\end{align}
\begin{exa} The first few terms of the partition function for rank $r=2$ and $\chi = 1$ BPS states is 
\begin{equation*}
Z^{BPS}_{r}(t) = - t^2(1 + 6t + 21t^2 + 61t^3 + 165t^4 + 426t^5 + ...)
\end{equation*} 
As we mentioned above we expect that the physics result from \cite{szabo} corresponds instead to the degenerate stability condition on the wall, namely
\begin{equation*}
Z^{BPS}_{r, Coulomb}(t) = r^2 Z^{BPS}_{r, degen}(t) = 1 - 2t + 7t^2 - 18t^3 + 47 t^4 -110 t^5 + 258 t^6 -568 t^7 + ...
\end{equation*}
\end{exa}
\subsection{Rank $r = 3$} Under the projection $\pi_{\leq 3}\!: \exp(\mathfrak{g})\to\exp(\mathfrak{g}/\mathfrak{g}_{> 2})$ we find
\begin{equation*}
\pi_{\leq 3}(rhs) = \pi_{\leq 3}(T_{0, 1})\prod^{\to}_{a \geq 1, 1 \leq b \leq 3} \pi_{\leq 3}(T^{\Omega(a,b)}_{(a, b)})
\end{equation*}
with
\begin{align*}
\pi_{\leq 3}(T^{\Omega(0, 1)}_{0,1}) &= \exp(-e_{\gamma}-\frac{1}{4}e_{2\gamma}-\frac{1}{9}e_{3\gamma}),\\
\pi_{\leq 3}(T^{\Omega(a, 1)}_{a,1}) &= \exp(-\Omega(a, 1)(e_{a\mu +\gamma} + \frac{1}{4}e_{2a\mu+2\gamma} + \frac{1}{9}e_{3a\mu+3\gamma})),\\
\pi_{\leq 3}(T^{\Omega(a, 2)}_{a, 2}) &= \exp(-\Omega(a,2)e_{a\mu + 2\gamma})\\
\pi_{\leq 3}(T^{\Omega(a, 3)}_{a, 3}) &= \exp(-\Omega(a,3)e_{a\mu + 3\gamma}).
\end{align*}
We can compute which terms $x$ with $\bra\mu, x \ket = 3$ appear in $\log(rhs)$ in the Lie algebra $\mathfrak{g}/\mathfrak{g}_{>3}$. These terms have a different form according to an ordered partition for the rank $r = 3$, namely $3, 2+1, 1+2, 1+1+1$, corresponding to the order of the Lie brackets involved. The type $3$ term is
\begin{equation*}
-\frac{1}{9}\sum_{a \geq 0}\Omega(a,1)e_{3a\mu + 3\gamma} - \sum_{a \geq 1}\Omega(a, 3)e_{a\mu + 3\gamma}.
\end{equation*} 
The type $2+1$ comprises 
\[\frac{1}{2}\sum_{a_1 < a_2}[-\frac{1}{4}\Omega(a_1, 1)e_{2 a_1 \mu + 2\gamma}, -\Omega(a_2, 1)e_{a_2\mu +\gamma}]\]
\[= \frac{1}{4}\sum_{a_1 < a_2}(a_1-a_2)\Omega(a_1, 1)\Omega(a_2, 1)e_{(2 a_1 + a_2)+3\gamma}\]
and
\[\frac{1}{2}\sum_{a_1 < 2 a_2}[-\Omega(a_1,2)e_{a_1\mu+2\gamma}, -\Omega(a_2,1)e_{a_2\mu + \gamma}]\]
\[= \frac{1}{2}\sum_{a_1 < 2 a_2}(-1)^{a_1-2a_2}(a_1 - 2a_2)\Omega(a_1,2)\Omega(a_2,1)e_{(a_1+a_2)\mu + 3\gamma}.\]
Similarly for type $1+2$ there are terms
\[\frac{1}{2}\sum_{a_1 < a_2}[-\Omega(a_1,1)e_{a_1\mu+\gamma},-\frac{1}{4}\Omega(a_2,1)e_{2 a_2\mu + 2\gamma}]\]
\[=\frac{1}{4}\sum_{a_1 < a_2}(a_1-a_2)\Omega(a_1,1)\Omega(a_2,1)e_{(a_1+2a_2)\mu + 3\gamma}\]
and
\[\frac{1}{2}\sum_{2a_1 < a_2}[-\Omega(a_1,1)e_{a_1\mu+\gamma},-\Omega(a_2, 2)e_{a_2\mu + 2\gamma}]\]
\[=\frac{1}{2}\sum_{2a_1 < a_2}(-1)^{2a_1-a_2}(2a_1-a_2)\Omega(a_1,1)\Omega(a_2,2)e_{(a_1 + a_2)\mu+3\gamma}.\]
For the type $1+1+1$ term recall the BCH formula up to order $3$ Lie brackets, 
\begin{equation*}
\log(\exp(X)\exp(Y)) = X + Y + \frac{1}{2}[X,Y] + \frac{1}{12}([X,[X,Y]]-[Y,[X,Y]]),
\end{equation*}
so we find contributions
\begin{align*}
\nonumber\frac{1}{4}\sum_{a_1 < a_2 < a_3}[[-\Omega(a_1,1)e_{a_1\mu + \gamma},-\Omega(a_2,1)&e_{a_2\mu+\gamma}],-\Omega(a_3,1)e_{a_3\mu+\gamma}]\\
\nonumber=-\frac{1}{4}\sum_{a_1 < a_2 < a_3}(-1)^{a_1-a_2}(-1)^{a_1 + a_2-2a_3}&(a_1-a_2)(a_1 + a_2 - 2a_3)\\
&\Omega(a_1,1)\Omega(a_2,1)\Omega(a_3,1)e_{(a_1+a_2+a_3)\mu+3\gamma},
\end{align*}
\begin{align*}
\nonumber\frac{1}{4}\sum_{a_1 < a_2 < a_3}[-\Omega(a_1,1)e_{a_1\mu + \gamma},[-\Omega(a_2,1)&e_{a_2\mu+\gamma},-\Omega(a_3,1)e_{a_3\mu+\gamma}]]\\
\nonumber=-\frac{1}{4}\sum_{a_1 < a_2 < a_3}(-1)^{a_2-a_3}(-1)^{2a_1 - a_2-a_3}&(a_1-a_2)(2a_1 - a_2 - a_3)\\
&\Omega(a_1,1)\Omega(a_2,1)\Omega(a_3,1)e_{(a_1+a_2+a_3)\mu+3\gamma},
\end{align*}
and
\begin{align*}
\frac{1}{12}\sum_{a_1 < a_2}\big([-\Omega(a_1,1)&e_{a_1\mu+\gamma},[-\Omega(a_1,1)e_{a_1\mu+\gamma},-\Omega(a_2,1)e_{a_2\mu+\gamma}]]\\
&-[-\Omega(a_2,1)e_{a_2\mu+\gamma},[-\Omega(a_1,1)e_{a_1\mu+\gamma},-\Omega(a_2,1)e_{a_2\mu+\gamma}]]\big)
\end{align*}
\begin{align*}
\nonumber = -\frac{1}{12}\sum_{a_1 < a_2}(-1)^{2(a_1-a_2)}(a_1-a_2)^2\big(&(\Omega(a_1,1))^2\Omega(a_2,1)e_{(2a_1+a_2)\mu + 3\gamma}\\
&+ \Omega(a_1,1)(\Omega(a_2,1))^2 e_{(a_1+ 2a_2)\mu + 3\gamma}\big).
\end{align*}
Summing over the previous terms we find the lengthy $r = 3$ KS identity
\begin{align}
\nonumber\Omega(a,3) &= (-1)^a\sum_{k \geq 1}3^{k-2}\frac{\chi^k}{k!}\sum_{\len(\n) = \len(\ii) = k, \n\cdot\ii = a}\frac{\prod n_l}{\prod i_l}\\
\nonumber&+\frac{1}{4}\sum_{a_1 < a_2, 2a_1+a_2 = a}(a_1-a_2)\Omega(a_1, 1)\Omega(a_2, 1)\\
\nonumber&+\frac{1}{2}\sum_{a_1 < 2a_2,a_1+a_2=a}(-1)^{a_1-2a_2}(a_1-2a_2)\Omega(a_1,2)\Omega(a_2,1)\\
\nonumber&+\frac{1}{4}\sum_{a_1 < a_2, a_1+2a_2 = a}(a_1-a_2)\Omega(a_1, 1)\Omega(a_2, 1)\\
\nonumber&+\frac{1}{2}\sum_{2a_1 < a_2, a_1+a_2 = a}(-1)^{2a_1-a_2}(2a_1-a_2)\Omega(a_1,1)\Omega(a_2, 2)\\
\nonumber&-\frac{1}{4}\sum_{a_1 < a_2 < a_3, a_1+a_2+a_3=a}(a_1-a_2)(a_1+a_2-2a_3)\Omega(a_1,1)\Omega(a_2,1)\Omega(a_3,1)\\
\nonumber&-\frac{1}{4}\sum_{a_1 < a_2 < a_3, a_1+a_2+a_3=a}(a_2-a_3)(2a_1-a_2-a_3)\Omega(a_1,1)\Omega(a_2,1)\Omega(a_3,1)\\
\nonumber&-\frac{1}{12}\sum_{a_1 < a_2,2a_1+a_2 = a}(a_1-a_2)^2(\Omega(a_1,1))^2\Omega(a_2,1)\\
\nonumber&-\frac{1}{12}\sum_{a_1 < a_2,a_1+2a_2 = a}(a_1-a_2)^2\Omega(a_1,1)(\Omega(a_2,1))^2\\
&-\frac{1}{9}\Omega(a/3, 1)
\end{align}
where it is understood that the last term only appears when $3\mid a$. As in the case of rank $r=2$ this gives an identity for $\bar\dt(a, 3) = \Omega(a, 3) + \frac{1}{9}\Omega(a/3, 1)$ where the last term only appears if $3\mid a$.
\subsection{Application to integrality} In the next section we will prove that the Joyce-Song invariants $\bar\dt(a, r)$ for $r \leq 3$ satisfy the above KS identities. Here we show how to deduce integrality of the BPS numbers for $r \leq 3$ from these identities. The best result towards integrality in general has been proved by Reineke \cite{rein}, but it does not seem to imply integrality for D0-D6 states counts, at least without further work.

Consider first the case $r = 2$. When $2 \nmid a$ we have $\Omega(a, 2) = \bar\dt(a, 2)$ which is integral by Joyce-Song theory since the class $(a,2)$ is primitive. Therefore we assume $2\mid a$. Going back to the $r = 2$ KS identity, notice that 
\begin{align*}
\sum_{a' < a'', a' + a'' = a}(a'-a'')\Omega(a', 1)\Omega(a'', 1) &= \sum_{a' < a'', a' + a'' = a}(a-2a'')\Omega(a', 1)\Omega(a'', 1)\\
&\equiv 0 \mod 2.
\end{align*}
So integrality of $\Omega(a, 2)$ follows if we can prove that
\begin{equation*}
\sum_{k \geq 1} 2^{k-2}\frac{\chi^k}{k!} \sum_{\operatorname{len}(\n) = \operatorname{len}(\ii) = k, \n\cdot\ii = a} \frac{\prod n_l}{\prod i_l} - \frac{1}{4}\Omega(a/2,1)
\end{equation*}
is an integer. This in turn is equivalent to 
\begin{equation*}
\sum_{k \geq 1}\frac{(2\chi)^k}{k!} \sum_{\operatorname{len}(\n) = \operatorname{len}(\ii) = k, \n\cdot\ii = a} \frac{\prod n_l}{\prod i_l} \equiv \Omega(a/2,1) \mod 4.
\end{equation*}
But notice that we can use the $r = 1$ KS identity to relate the left hand side of the above congruence to the McMahon function $M(t)$, namely the left hand side is just the coefficient of $t^a$ in the formal power series $M(t)^{2\chi} = \prod_{n\geq 1}(1-t^n)^{-2 n \chi}$. The right hand side is the coefficient of $t^{a/2}$ in the formal power series $M(-t)^{\chi} = \prod_{n\geq 1}(1-(-t)^n)^{-n \chi}$. So the $r = 1, 2$ KS identities together reduce integrality to
\begin{lem} For $2 \mid a$
\[[t^a]M(t)^{2\chi} \equiv (-1)^{a/2}[t^{a/2}]M(t)^{\chi} \mod 4.\]
\end{lem}
\begin{proof} We use the identity for Euler products
\begin{equation}
[t^a]\prod_{n\geq 1}(1-t^n)^{-c_n}) = \sum_{\p\vdash a}\prod_{i\geq 1}\binom{c_n-1 + p_{i}-p_{i+1}}{p_{i}-p_{i+1}}
\end{equation}
where (in contrast to the rest of the paper) the sum is over partitions rather than ordered partitions. We learned of this representation from \cite{rein} Lemma 5.3. In our case this gives
\begin{equation*}
[t^a]M(t)^{2\chi} = \sum_{\p \vdash a}\prod_{i \geq 1}\binom{2 i \chi - 1 + p_{i}-p_{i+1}}{p_{i}-p_{i+1}}
\end{equation*}
and 
\begin{equation*}
[t^{a/2}]M(t)^{\chi} = \sum_{\q \vdash a/2}\prod_{j \geq 1}\binom{j \chi - 1 + q_{i}-q_{i+1}}{q_{i}-q_{i+1}}.
\end{equation*}
Note that 
\begin{equation*}
\binom{2 i \chi -1 + \xi}{\xi} \equiv 0 \mod 2 \text{ for } \xi \equiv 1 \mod 2,
\end{equation*}
so the restriction of the first sum to partitions which contain parts of each parity is $\equiv 0 \mod 4$. On the other hand if the partition only contains odd parts, there must be an even number of them since (as $a$ is even) and then the sum is still $\equiv 0 \mod 4$ by the congruence
\begin{equation*}
\binom{2 i \chi - 1 + \xi}{\xi} \equiv 0 \mod 4 \text{for } i \equiv 0 \mod 2 \text{ and } \xi \equiv 1 \mod 2
\end{equation*} 
applied when $\xi$ is the last part of the partition. It remains to show that for a partition $\p$ with even parts
\begin{equation*}
\prod_{i \geq 1}\binom{2 i \chi - 1 + p_{i}-p_{i+1}}{p_{i}-p_{i+1}} \equiv (-1)^{a/2}\prod_{i \geq 1}\binom{i \chi - 1 + p_{i}/2-p_{i+1}/2}{p_{i}/2-p_{i+1}/2} \mod 4.
\end{equation*}
But this follows from
\begin{equation*}
\binom{2 i \chi - 1 + \xi}{\xi} \equiv (-1)^{\xi/2}\prod_{i \geq 1}\binom{i \chi - 1 + \xi/2}{\xi/2} \mod 4 \text{ for } \xi \equiv 0 \mod 2
\end{equation*}
which can be proved by induction.
\end{proof}

Similarly in the $r = 3$ case we already know that the Joyce-Song invariants are integral for primitive classes, so we assume $3 \mid a$. We need an analogue of the above lemma.
\begin{lem}For $3 \mid a$
\[[t^a]M(-t)^{3\chi} \equiv [t^{a/3}]M(-t)^{\chi} \mod 9.\]
\end{lem}  
\begin{proof} As before, summing over partitions (not ordered partitions, as we will do in the rest of the paper) 
\begin{equation*}
[t^a]M(-t)^{3\chi} = \sum_{\p \vdash a}\prod_{i \geq 1}\binom{3 i \chi-1 + p_{i}-p_{i+1}}{p_{i}-p_{i+1}}
\end{equation*}
and
\begin{equation*}
\binom{3 i \chi -1 + \xi}{\xi} \equiv 0 \mod 3 \text{ for } \xi \equiv 1, 2 \mod 3,
\end{equation*}
so modulo 9 we only need to sum over partitions whose parts all have the same residue modulo 3. The cases when this common residue is $1$ or $2$ vanish $\mod 9$ since in either case the number of parts must be divisible by $3$ and we can apply to the last part the congruence
\begin{equation*}
\binom{3 i \chi -1 + \xi}{\xi} \equiv 0 \mod 9 \text{ for } i \equiv 0,\,\,\, \xi \equiv 1, 2 \mod 3.   
\end{equation*}  
So we reduce to partitions all whose parts are $\equiv 0 \mod 3$. The result now follows from
\begin{equation*}
(-1)^{\xi}\binom{3 i \chi -1 + \xi}{\xi} \equiv (-1)^{\xi/3}\binom{i \chi -1 + \xi/3}{\xi/3} \mod 9 \text{ for } \xi \equiv 0 \mod 3
\end{equation*}
which can be proved by induction.
\end{proof}
One can then show that the result would follow from the integrality of
\begin{align*}
&-\frac{1}{12}\sum_{a_1 < a_2,2a_1+a_2 = a}(a_1-a_2)^2(\Omega(a_1,1))^2\Omega(a_2,1)\\
&-\frac{1}{12}\sum_{a_1 < a_2,a_1+2a_2 = a}(a_1-a_2)^2\Omega(a_1,1)(\Omega(a_2,1))^2
\end{align*}
But this follows since we are assuming $3 \mid a$ so e.g. in the first term
\begin{equation*}
a_1 - a_2 = 3 a_1 - a = 3\big(a_1-\frac{a}{3}\big) 
\end{equation*}
and similarly for the second.
\section{Joyce-Song side}\label{JS_sect}
In this section we use Joyce-Song theory for precisely the same wall-crossing described in the introduction. The tilted category $\mathcal{A}'$ satisfies again the assumptions of the theory and the Joyce-Song invariants do not change until the phase of $\mu$ crosses that of $\gamma$. One can check directly that for $\phi(\mu) > \phi(\gamma)$ the Joyce-Song invariants, which we call $\bar\dt^{-}$, vanish for all mixed classes. The general wall-crossing formula in JS theory (see \cite{ks} Section 6.5) is 
\begin{align*}
\bar\dt(\alpha) = \sum_{n \geq 1}\sum_{\alpha_1, \dots, \alpha_n}\frac{(-1)^{n-1}}{2^{n-1}}&\U(\alpha_1, \dots, \alpha_n; \phi_{\mp})\\
&\cdot\sum_{\Upsilon}\prod_{\{i \to j\}\subset \Upsilon^{1}}(-1)^{\bra\alpha_i, \alpha_j\ket}\bra\alpha_i, \alpha_j\ket\prod_k\bar\dt^{-}(\alpha_k)
\end{align*}
where we are summing over effective decompositions of the $K$-theory class $\alpha$ and ordered trees respectively, but we will only explain this in this section for the very special case of D0-D6 states. In general it is not known if the $\bar\dt$ invariants satisfy the Kontsevich-Soibelman wall-crossing formula and only the more complicated Joyce-Song identity has been rigorously established. 
\subsection{Rank $r = 1$} 
\subsubsection{Decompositions and partitions} Fix a $K$-theory class $\alpha = \gamma + a\mu$ and let 
\begin{equation}\label{decompGen}
\alpha = \alpha_1 + \dots + \alpha_n
\end{equation}
be an ordered \emph{decomposition} into effective classes; this corresponds to a 2D ordered partition of the integer vector $(a, 1)$. This decomposition a priori gives a contribution to $\bar\dt(\alpha)$ via Joyce-Song wall-crossing, which is given by a multiple of the $\bar\dt^{-}$ invariant of the 2D partition, $\prod_{k}\bar\dt^{-}(\alpha_k)$. However $\bar\dt^{-}(\beta)$ vanishes for `mixed classes' $\bra \beta, \gamma \ket, \bra\beta, \mu\ket \neq 0$. Thus we can effectively restrict to summing over pairs $(\p, i)$ given by an \emph{ordered partition} $\p$ for $a$ of length $n-1$ and an integer $i = 1, \dots, n$ denoting the place of the (unique) summand $\gamma$ in the decomposition, so that the decomposition of $\alpha$ above looks like
\begin{equation}\label{decomp}
\gamma + a\mu = p_1\mu + \dots + p_{i-1}\mu + \gamma + p_{i}\mu + \dots + p_{n-1}\mu
\end{equation}   
(writing $\p = (p_1, \dots, p_{n-1})$). We write $\p \vdash a$ for an ordered partition of $a$. 
\subsubsection{S symbols} Let us denote by $\phi_{\mp} = \arg\circ Z_{\mp}$ the phase functions with respect to the two different central charges $Z_{\mp}$. We need to compute Joyce's $\s$ symbol (see e.g. \cite{joy} Definition 3.12)
\[\s(\p, i) = \s(p_1\mu , \dots , p_{i-1}\mu , \gamma , p_{i}\mu , \dots , p_{n-1}\mu; \phi_{\mp}).\]
Its value is determined by a set of `see-saw' inequalities (the inequalities (a) and (b) in \cite{joy} Definition 3.12), which say roughly that $\s$ is an ordering operator. Suppose $i > 2$. Then since 
\begin{align*}
\phi_{-}(p_1 \mu) &= \phi_{-}(p_2 \mu),\\
\phi_{+}(p_1 \mu) &< \phi_{+}(\gamma + p_2\mu + \dots + p_{n-1}\mu)
\end{align*}
the see-saw inequalities do not hold and $\s = 0$. For $i = 2$ the see-saw inequalities do hold since
\begin{align*}
\phi_{-}(p_1 \mu) &> \phi_{-}(\gamma),\\
\phi_{+}(p_1 \mu) &\leq \phi_{+}(\gamma + p_1\mu + \dots + p_{n-1}\mu);\\
\phi_{-}(\gamma) &< \phi_{-}(p_2\mu),\\
\phi_{+}(\gamma + p_1\mu) &> \phi_{+}(p_2\mu + \dots + p_{n-1}\mu);\\
\phi_{-}(p_k\mu) &= \phi_{-}(p_{k+1}\mu),\\
\phi_{+}(\gamma + p_1\mu + \dots + p_k\mu) &> \phi_{+}(p_{k+1}\mu + \dots + p_{n-2}\mu)
\end{align*}
for $k = 2, \dots, n-2$. When the see-saw inequalities hold $\s$ is $(-1)^{\# \text{adjecent} (\leq, >) \text{pairs}}$, which gives
\begin{equation}
\s(\p, 2) = (-1)^{n-2}.
\end{equation}
Similarly for $i = 1$ the see-saw inequalities hold since
\begin{align*}
\phi_{-}(\gamma) &< \phi_{-}(p_1\mu),\\
\phi_{+}(\gamma) &> \phi_{+}(p_1\mu + \dots + p_{n-1}\mu);\\
\phi_{-}(p_k\mu) &= \phi_{-}(p_{k+1}\mu),\\
\phi_{+}(\gamma + p_1\mu + \dots + p_{k}\mu) &< \phi_{+}(p_{k+1}\mu + \dots + p_{n-1}\mu)
\end{align*}
for $k = 1, \dots, n-2$, giving 
\begin{equation}
\s(\p, 1) = (-1)^{n-1}.
\end{equation}
\subsubsection{U symbols} Consider again the decomposition \eqref{decomp}. We can obtain a new one of the same form by partitioning the \emph{head} and \emph{tail} sets 
\[\{p_1\mu, \dots, p_{i-1}\mu\}, \{p_i \mu, \dots, p_{n-1}\mu\}\]
according to partions $\q', \q''$ of $i-1, n-i$ and taking the partial sums of $\q', \q''$. We call this a \emph{contraction} $(\p', i)$ of the decomposition $(\p, i)$. A contraction carries a \emph{weight} 
\[\frac{1}{\prod_k q'_k!\prod_l q''_l!}.\] 

The ideal sheaves $K$-theory classes $(a, 1)$ are primitive. In this situation Joyce's $\U(\p, i)$ symbol (\cite{joy} Definition 3.12) reduces to the weighted sum over all contractions of $(\p, i)$ with non-vanishing $\s$ symbol.  

Suppose $i > 1$. Then the only choice for $\q'$ is the trivial partition of $i-1$ (i.e. we must contract all of $\{p_1\mu, \dots, p_{i-1}\mu\}$ to the single class $(p_1 + \dots + p_{i-1})\mu$) with weight $(i-1)!^{-1}$. On the other hand we can contract the tail with an arbitrary $\q \vdash n-i$ with weight $(\prod_k q_k!)^{-1}$. The contracted decomposition is of type $(\p', 2)$, has length $2 + \len(\q)$ and thus $\s$ symbol $(-1)^{\len(\q)}$. 

For $i = 1$ instead the head is empty and the $\q$-contracted decomposition has type $(\p', 1)$, length $1+ \len(\q)$ and thus $\s = (-1)^{\len(\q)}$. So we see that for $i \geq 1$
\begin{equation}
\U(\p, i) = \frac{1}{(i-1)!}\sum_{\q \vdash n-i} \frac{(-1)^{\len(q)}}{\prod_k q_k!}.
\end{equation} 
The result is \emph{independent} of $\p$. Next notice the identity
\begin{equation*}
\sum_{\q \vdash s}\frac{(-1)^{\len(\q)}}{\prod_l q_l!} = \frac{(-1)^s}{s!}
\end{equation*}
which is easily proved by induction, 
\begin{align*}
\sum_{\q \vdash s}\frac{(-1)^{\len(\q)}}{\prod_l q_l!} &= -\sum^{s}_{q_1 = 1}\frac{1}{q_1!}\sum_{\q'\vdash s-q_1} \frac{(-1)^{\len(\q')}}{\prod_l q'_l!}\\
&= -\sum^s_{q_1 = 1}\frac{(-1)^{s-q_1}}{q_1!(s-q_1)!} = \frac{(-1)^s}{s!}.
\end{align*}
Using this identity we find for $i \geq 1$
\begin{equation}
\U(\p, i) = \frac{(-1)^{n-i}}{(i-1)!(n-i)!}.
\end{equation}
Notice that in particular
\begin{equation*}
\sum^n_{i = 1}(-1)^i\U(\p, i) = (-1)^n\frac{2^{n-1}}{(n-1)!}.
\end{equation*}
\subsubsection{Sums over trees} The wall-crossing for the decomposition \eqref{decompGen} carries a sum over trees factor
\[\sum_{\Upsilon} \prod_{\{k \to l\}\subset \Upsilon^1}(-1)^{\bra\alpha_k, \alpha_l\ket}\bra\alpha_k, \alpha_l\ket\]
which is especially simple for $r = 1$. Since $\bra p_k \mu, p_l \mu\ket = 0$ the only ordered tree which gives a non-vanishing factor is the unique ordered tree rooted at $i$ with leaves labelled by $1, \dots i-1, i+1, \dots n$. The factor is then
\begin{equation*}
\prod_{k}(-1)^{p_k}\prod^{i-1}_{l = 1}\bra p_l\mu, \gamma\ket \prod^{n-1}_{l = i}\bra \gamma,p_l\mu\ket = (-1)^{n+i}\prod_k (-1)^{p_k}p_k = (-1)^a(-1)^{n + i}\prod_k p_k.
\end{equation*}
\subsubsection{$\bar\dt^{-}$ of a partition} Since $\bar\dt^{-}(\gamma) = 1$ this is simply the product 
\[\prod_{k}\dt^{-}(p_k \mu),\] 
in particular it only depends on the unordered partition underlying $\p$. Thus we compute
\begin{equation*}
\bar\dt^{-}(\p,i) = (-1)^{n-1}\chi^{n-1}\prod^{n-1}_{k = 1}\left(\sum_{m\mid p_k}\frac{1}{m^2}\right).
\end{equation*} 
\subsubsection{$r = 1$ wall-crossing} We can now write down the rank $r = 1$ wall-crossing formula explicitely in terms of ordered partitions for integers, 
\begin{align*}
\bar\dt(a,1) &= (-1)^a\sum_{n \geq 2}(-1)^n(-\chi)^{n-1}\frac{(-1)^{n-1}}{2^{n-1}}\left(\sum^n_{i = 1}(-1)^i\U(\p, i)\right)\\
&\left(\sum_{\p \vdash a, \len(p) = n-1}\prod_k p_k\left(\sum_{m\mid p_k}\frac{1}{m^2}\right)\right)\\
&=(-1)^a\sum_{n \geq 2}\frac{\chi^{n-1}}{(n-1)!}\left(\sum_{\p \vdash a, \len(p) = n-1}\prod_k p_k\left(\sum_{m\mid p_k}\frac{1}{m^2}\right)\right).
\end{align*}
This can be compared directly with the KS wall crossing. Rearranging we find
\begin{equation}
\sum_{\p \vdash a, \len(\p) = k}\prod_l p_l\left(\sum_{m\mid p_l}\frac{1}{m^2}\right) = \sum_{\operatorname{len}(\n) = \operatorname{len}(\ii) = k, \n\cdot\ii = a} \frac{\prod n_l}{\prod i_l}. 
\end{equation}
which proves the required equivalence.
\subsection{$r = 2$}
\subsubsection{Decompositions} The rank $r = 2$ wall crossing formula contains a copy of the $r = 1$ case, up to scale, given by ordered decompositions of the form
\begin{equation}
2\gamma + a\mu = p_1 \mu + \dots + p_{i-1}\mu + 2\gamma + p_{i}\mu + \dots p_{n-1}\mu. 
\end{equation}
This is because for decompositions of the form above it makes no difference if the $K$-theory class is not primitive: $\U$ remains the sum of $\s$ over all possible contractions.
The factor 
\[\prod_k(-1)^{p_k}p_k\left(\sum_{m\mid p_k}\frac{1}{m^2}\right) = (-1)^a\prod_k p_k\left(\sum_{m\mid p_k}\frac{1}{m^2}\right)\]
in the $\chi^{n-1}$ coefficient of the $r = 1$ formula must be replaced by
\[2^{n-1}\prod_k p_k\frac{1}{4}\left(\sum_{m\mid p_k}\frac{1}{m^2}\right)\]
accounting for products $\pm(-1)^{\bra p_k\mu, 2\gamma\ket}\bra p_k\mu, 2\gamma\ket$ and $\bar\dt^{-}(2\gamma) = \frac{1}{4}$, giving
\begin{align*}
\bar\dt(a,2) &= \sum_{n \geq 2}2^{n-3}\frac{\chi^{n-1}}{(n-1)!}\left(\sum_{\p \vdash a, \len(p) = n-1}\prod_k p_k\left(\sum_{m\mid p_k}\frac{1}{m^2}\right)\right)\\
&+ \text{contribution of new decompositions.}
\end{align*}
The first term coincides precisely with the first term of the rank $r = 2$ KS formula (we may call this the `scaling' behaviour of both the KS and JS formulae in the D0-D6 case). The residual contribution comes from decompositions of the form
\begin{equation}\label{decomp2}
2\gamma + a \mu = p_1\mu + \dots + p_{i-1}\mu + \gamma + p_i\mu + \dots + p_{j-2}\mu + \gamma + p_{j-1}\mu + \dots + p_{n-2}\mu 
\end{equation}
with copies of $\gamma$ sitting at places $1\leq i < j \leq n$, which we denote by $(\p, i, j)$ where $\p$ is a length $n-2 \geq 1$ ordered partition of $a$. In the rest of this section we compute this residual contribution.
\subsubsection{Sum over trees} For fixed values of indexes $i, j$, $1 \leq i < j \leq n$ choose a special integer $l \in \{1, \dots , n\}\setminus\{i, j\}$; then choose possibly empty subsets of 
\[\{1, \dots , i-1\}\setminus\{l\}, \{i+1, \dots , j-1\}\setminus\{l\}, \{j+1, \dots , n\}\setminus\{l\}\]
with cardinality $h, m, t$ respectively. These choices give rise to a well defined ordered tree rooted at $i, j$ by connecting the chosen sets to the vertex labelled $i$, the special vertex $l$ to both $i, j$ and the remaining edges to $j$. Two such trees can be distinguished by their Pr\"ufer code, and all admissible trees for \eqref{decomp2} are of this form. 

A fixed tree contributes to the wall-crossing formula by a common factor $\prod_k(-1)^{p_k}p_k = (-1)^a \prod_k(-1)^{p_k}p_k$ times by a factor specific to the tree. Suppose first $l \in \{1, \dots, i-1\}$; then this factor is 
\begin{equation*}
(-1)^{\#\{\text{ edges outgoing from } i \text{ or } j\}}p_l = (-1)^{m}(-1)^{n-j}p_l.
\end{equation*}
There are $2^{i-2}2^{n-j}\binom{j-i-1}{m}$ trees for such fixed $l$. For $l \in \{i + 1, j-1\}$ the factor is
\begin{equation*}
(-1)^{\#\{\text{ edges outgoing from } i \text{ or } j\}}p_{l-1} = (-1)^{m+1}(-1)^{n-j}p_{l-1}, 
\end{equation*}
and there are $2^{i-1}2^{n-j}\binom{j-i-2}{m}$ trees for such fixed $l$. Finally $l \in \{j+1, \dots, n\}$ gives a factor
\begin{equation*}
(-1)^{\#\{\text{ edges outgoing from } i \text{ or } j\}}p_{l-2} = (-1)^{m}(-1)^{n-j+1}p_{l-2}
\end{equation*}
for $2^{i-1}2^{n-j-1}\binom{j-i-1}{m}$ trees. Thus the sum over graphs turns out to be
\[
\sum_{\Upsilon} \prod_{\{k \to l\}\subset \Upsilon^1}(-1)^{\bra\alpha_k, \alpha_l\ket}\bra\alpha_k, \alpha_l\ket = \] 
\[\prod_{k}(-1)^{p_k}p_k \left(2^{i-2}2^{n-j}(-1)^{n-j}\sum^{i-1}_{l = 1}p_{l}\sum^{j-i-1}_{m = 0}(-1)^{m}\binom{j-i-1}{m}\right.\]
\[+2^{i-1}2^{n-j}(-1)^{n-j}\sum^{j-1}_{l = i+1}p_{l-1}\sum^{j-i-2}_{m = 0}(-1)^{m+1}\binom{j-i-2}{m}\]
\[\left. + 2^{i-1}2^{n-j-1}(-1)^{n-j+1}\sum^{n}_{l = j+1}p_{l-2}\sum^{j-i-1}_{m = 0}(-1)^m\binom{j-i-1}{m}\right).\]
By the binomial theorem this equals
\[\left\{\begin{matrix}
&(-1)^n(-1)^{i+1}2^{n-3}(-1)^a\prod_{k}p_k(\sum^{i-1}_{l=1}p_l - \sum^{n-2}_{l=i}p_{l})&\text{ if }& j = i+1,&\\
&(-1)^n(-1)^{i+1}2^{n-3}(-1)^a\prod_{k}p_k p_i  &\text{ if }& j = i+2,&\\
&0 &\text{ otherwise.}&  
\end{matrix}\right.\]
(recall $n \geq 3$). The upshot of this is that among decompositions \eqref{decomp2} the only that can possibly contribute to the wall-crossing are those with $(i, j)$ as above.
\subsubsection{S and U} We only need to compute $\U$ of the decompositions with non-vanishing $\sum_{\Upsilon}$ factor. Notice first that as in the $r=1$ case the $\s$ symbol of a partition can only be non-vanishing if the first copy of $\gamma$ lies in the first or second place. As in the primitive case $\U(\p, i, j)$ contains a `first order' term  which is the weighted sum of $\s$ over admissible contractions of $\p$. For an arbitrary $\p$ we must contract the head $\{p_1\mu, \dots, p_{i-1}\mu\}$ to the singleton $\{(p_1 + \dots + p_{i-1})\mu\}$. 

Suppose first $j = i+1$. Then contracting the head to a singleton plus contracting the tail using a partition $\q$ has $\s$ symbol 
\[(-1)^{\len(\q)+1}\delta_{p_1 + \dots + p_{i-1} < p_i + \dots + p_{n-2}}.\]
If we also contract the couple $\{\gamma, \gamma\}$ (with weight $1/2$) the $\s$ symbol becomes $(-1)^{\len(\q)}$. The `first order' $\U$ symbol for $j = i+1$ is therefore
\[\frac{1}{(i-1)!}\sum_{\q \vdash n-i-1}\frac{(-1)^{\len(\q)}}{\prod_l q_l!}\left(-\delta_{p_1 + \dots + p_{i-1} < p_i + \dots + p_{n-2}} + \frac{1}{2}\right)\]
\[= \frac{1}{2}\frac{(-1)^{n-i-1}}{(i-1)!(n-i-1)!}(\delta_{p_1 + \dots + p_{i-1} \geq p_i + \dots + p_{n-2}}-\delta_{p_1 + \dots + p_{i-1} < p_i + \dots + p_{n-2}}).\]
For $j = i + 2$ the corresponding `first order' term is
\[\frac{1}{(i-1)!}\sum_{\q \vdash n-i-1}\frac{(-1)^{\len(\q)}}{\prod_l q_l!}(\delta_{p_1 + \dots + p_{i-1} < p_i + \dots + p_{n-2}}\cdot\delta_{p_1 + \dots + p_{i} \geq p_{i+1} + \dots + p_{n-2}})\]
\[=\frac{(-1)^{n-i-1}}{(i-1)!(n-i-1)!}(\delta_{p_1 + \dots + p_{i-1} < p_i + \dots + p_{n-2}}\cdot\delta_{p_1 + \dots + p_{i} \geq p_{i+1} + \dots + p_{n-2}}).\]
In both cases when $2\mid a$ there is also a `second order' term. For $j = i+1$ it is 
\[-\frac{1}{2}\frac{1}{(i-1)!}\sum_{\q \vdash n-i-1}\frac{(-1)^{\len(\q)}}{\prod_l q_l!}\delta_{p_1 + \dots + p_{i-1} = p_i + \dots + p_{n-2}}\]
\[=-\frac{1}{2}\frac{(-1)^{n-i-1}}{(i-1)!(n-i-1)!} \delta_{p_1 + \dots + p_{i-1} = p_i + \dots + p_{n-2}},\]
while for $j=i+2$ we get
\[-\frac{1}{2}\frac{1}{(i-1)!}\sum_{\q \vdash n-i-1}\frac{(-1)^{\len(\q)}}{\prod_l q_l!}(-\delta_{p_1 + \dots + p_{i} = p_{i+1} + \dots + p_{n-2}} + \delta_{p_1 + \dots + p_{i-1} = p_{i} + \dots + p_{n-2}})\]
\[=\frac{1}{2}\frac{(-1)^{n-i-1}}{(i-1)!(n-i-1)!}(\delta_{p_1 + \dots + p_{i} = p_{i+1} + \dots + p_{n-2}} - \delta_{p_1 + \dots + p_{i-1} = p_{i} + \dots + p_{n-2}}).\]  
\subsubsection{$r = 2$ wall-crossing} Recall that our aim is to compare the `residual contribution' given in Joyce-Song theory by decompositions of the form \eqref{decomp2} with the corresponding term in the $r=2$ KS formula, namely
\[\frac{(-1)^a}{2}\sum_{a' < a'', a'+a''=a}(a'-a'')\Omega(a', 1)\Omega(a'', 1).\]
By the above discussion it is enough to sum over $\p$ and $i$ since $j$ is either $i+1$ or $i+2$, and the $\frac{(-1)^{n-1}}{2^{n-1}}\U$ factor over such a sum over decomposition equals
\[\frac{(-1)^a}{2}\frac{(-1)^n}{4}\frac{1}{(i-1)!(n-i-1)!}\prod_{k}p_k\left(\sum_{m \mid p_k}\frac{1}{m^2}\right)\]
\begin{align*}
\nonumber\cdot\,\big(\big(\sum^{i-1}_{l=1}p_l - \sum^{n-2}_{l=i}p_{l}\big)&(\delta_{p_1 + \dots + p_{i-1} < p_i + \dots + p_{n-2}}-\delta_{p_1 + \dots + p_{i-1} \geq p_i + \dots + p_{n-2}})\\
\nonumber&- p_i (2\delta_{p_1 + \dots + p_{i-1} < p_i + \dots + p_{n-2}}\cdot\delta_{p_1 + \dots + p_{i} \geq p_{i+1} + \dots + p_{n-2}}\\
&- \delta_{p_1 + \dots + p_{i-1} = p_{i} + \dots + p_{n-2}} + \delta_{p_1 + \dots + p_{i} = p_{i+1} + \dots + p_{n-2}})\big).
\end{align*}
Notice that the second order term for $\U$ when $j = i+1$ is only nonzero when $\sum^{i-1}_{l=1}p_l = \sum^{n-2}_{l=i}p_{l}$, hence it gives no contribution in the formula above.  

Now sum over all $\p, i$ and compare to the KS term. The second factor in the formula above acts as on ordering operator, giving the sum over $a' < a''$. This can be seen using the fact that for a fixed partition $\p$ there exists a unique $i$ with
\[\delta_{p_1 + \dots + p_{i-1} < p_i + \dots + p_{n-2}}\cdot\delta_{p_1 + \dots + p_{i} \geq p_{i+1} + \dots + p_{n-2}} = 1.\]
The first factor equals the sum over all product $\Omega(a'), \Omega(a'')$ by the usual rearrangement
\begin{equation*}
\sum_{\p \vdash a', \len(\p) = k}\prod_l p_l\left(\sum_{m\mid p_l}\frac{1}{m^2}\right) = \sum_{\operatorname{len}(\n) = \operatorname{len}(\ii) = k, \n\cdot\ii = a'} \frac{\prod n_l}{\prod i_l}, 
\end{equation*}
(same for $a''$), and the $r=1$ KS wall-crossing i.e.
\[\Omega(a', 1) = (-1)^{a'}\sum_{k \geq 1} \frac{\chi^k}{k!} \sum_{\operatorname{len}(\n) = \operatorname{len}(\ii) = k, \n\cdot\ii = a'} \frac{\prod n_l}{\prod i_l},\]
(same for $a''$).  
\subsection{$r = 3$} Exactly as for $r = 2$ case there is a copy of the rank $r = 1$ KS formula, up to scaling $\gamma$ to $3 \gamma$, contributing
\begin{equation*}
(-1)^a \sum_{n \geq 2}3^{n-1}\frac{\chi^{n-1}}{(n-1)!}\left(\sum_{\p\vdash a, \len(p)=n-1}\prod_k p_k\,\frac{1}{9}\left(\sum_{m\mid p_k} \frac{1}{m^2}\right)\right)
\end{equation*}
which can be identified with the term 
\begin{equation*}
(-1)^a\sum_{k \geq 1}3^{k-2}\frac{\chi^k}{k!}\sum_{\len(\n) = \len(\ii) = k, \n\cdot\ii = a}\frac{\prod n_l}{\prod i_l}
\end{equation*}
in the $r = 3$ KS formula.

Let us now consider the case when exactly $2$ copies of $\gamma$ appear in the decomposition, or in other words decompositions $(\p, 2_{i} + 1_{j}), (\p, 1_{i}+2_{j})$ for $i < j$. It should be clear that both cases are very close, up to scale, to the decompositions studied for the $r = 2$ case. One can go thorugh all of the previous subsection, treating the first or second copy of $\gamma$ as a `variable' which can be rescaled to $2\gamma$, without any additional changes, until we reach the very last paragraph where the $r = 1$ formula for $\Omega(a', 1), \Omega(a'', 1)$ is used. This must now be replaced with the corresponding $r = 2$ formula for $\Omega(a', 2)$ which gives
\begin{align*}
\nonumber&+\frac{1}{2}\sum_{a_1 < 2a_2,a_1+a_2=a}(-1)^{a_1-2a_2}(a_1-2a_2)\Omega(a_1,2)\Omega(a_2,1)\\
\nonumber&-\frac{1}{4}\sum_{a_1 < a_2 < a_3, a_1+a_2+a_3=a}(a_1-a_2)(a_1+a_2-2a_3)\Omega(a_1,1)\Omega(a_2,1)\Omega(a_3,1)\\
&+\frac{1}{4}\sum_{a_1 < a_2, 2a_1+a_2 = a}(a_1-a_2)\Omega(a_1, 1)\Omega(a_2, 1)\\
\end{align*}  
and for $\Omega(a'', 2)$, giving
\begin{align*}
\nonumber&+\frac{1}{2}\sum_{2a_1 < a_2, a_1+a_2 = a}(-1)^{2a_1-a_2}(2a_1-a_2)\Omega(a_1,1)\Omega(a_2, 2)\\
\nonumber&-\frac{1}{4}\sum_{a_1 < a_2 < a_3, a_1+a_2+a_3=a}(a_2-a_3)(2a_1-a_2-a_3)\Omega(a_1,1)\Omega(a_2,1)\Omega(a_3,1)\\
&+\frac{1}{4}\sum_{a_1 < a_2, a_1+2a_2 = a}(a_1-a_2)\Omega(a_1, 1)\Omega(a_2, 1)\\
\end{align*}
in the $r = 3$ KS identity.

It remains to consider the `genuine' new decompositions, i.e. those of the form $(\p, 1_i + 1_j + 1_k)$ for $i < j < k$. We expect that these contribute
\begin{align*}
\nonumber&-\frac{1}{12}\sum_{a_1 < a_2,2a_1+a_2 = a}(a_1-a_2)^2(\Omega(a_1,1))^2\Omega(a_2,1)\\
&-\frac{1}{12}\sum_{a_1 < a_2,a_1+2a_2 = a}(a_1-a_2)^2\Omega(a_1,1)(\Omega(a_2,1))^2.
\end{align*}
This can be shown be splitting the sum over graphs over the two different types of graphs, those with one `cap'
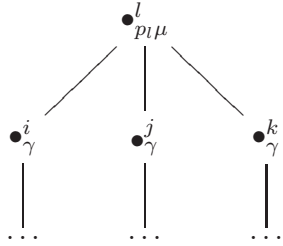
\begin{center}
\centerline{
\xymatrix{
 &                      & {\bullet^{l}_{p_l\mu}}\ar@{-}[dl]\ar@{-}[d]\ar@{-}[dr] &                    &\\
 & {\bullet^i_{\gamma}}\ar@{-}[d]  & {\bullet^j_{\gamma}}\ar@{-}[d]                       & {\bullet^k_{\gamma}}\ar@{-}[d]&\\
 &   \dots              &   \dots                                   &    \dots           &\\
}}
\end{center}
which account for terms which are square terms $\pm p^2_l$ (times by the usual common factor for trees), and
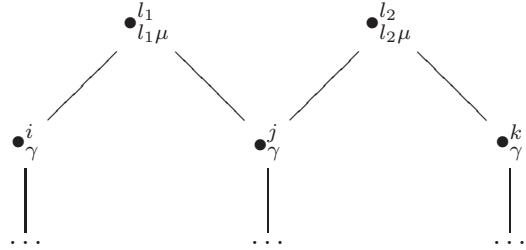
\begin{center}
\centerline{
\xymatrix{
 &                      &   {\bullet^{l_1}_{l_1\mu}}\ar@{-}[dl]\ar@{-}[dr] &                    & {\bullet^{l_2}_{l_2\mu}}\ar@{-}[dl]\ar@{-}[dr]& &\\
 & {\bullet^{i}_{\gamma}}\ar@{-}[d]  &                                   & {\bullet^{j}_{\gamma}}\ar@{-}[d]&                                & {\bullet^{k}_{\gamma}}\ar@{-}[d] &\\
 &   \dots              &                                   &    \dots           &                                & \dots &\\
}}
\end{center}
for $l_1 < l_2$, giving double products $\pm 2 p_{l_1} p_{l_2}$. This can be shown by computing directly as in the case of one-rooted graphs.
\subsection{Induction} We wish to sketch an inductive argument for the KS identities (but not integrality) for arbitrary rank $r$. 

Suppose we wish to prove that the Joyce-Song invariants satisfy the KS identity for rank $r$. We fix an ordered partition $\q\vdash r$ and places $\ii = i_1, \dots, i_{\len(\q)}$, corresponding to decompositions for $a \mu + r\gamma$ of type $(\p, \q, \ii)$. 

Suppose now that $\len(\q) < r$. By induction, we have a formula for the contributions of such partitions to the JS side when at least one of the parts equals $1$. And we also know inductively the scaling behaviour of this formula by replacing this part ($= 1$) with a multiple $q$ (here we use that the admissible trees are those with less than $r$ vertexes labelled by multiples of $\gamma$). As in passing from $r = 2$ to $r = 3$, this rescaling is given by replacing a factor $\Omega(a_k, 1)$ with the rank $q$ component of the $lhs$ in the KS wall-crossing formula. And as before we substitute for this the corresponding term on the $rhs$, which is given inductively in terms of products of $\Omega(a_k, q_l)$ for $q_l < q$. 

This procedure gives all the terms in the rank $r$ component of the $rhs$ of the KS wall-crossing, \emph{except} those arising from the order $r$ correction in the BCH formula for $\log(\exp(X)\exp(Y))$. The order $r$ correction to $\log(\exp(X)\exp(Y))$ has the form 
\begin{equation*}
\sum_{n > 0} \frac{(-1)^{n-1}}{n}\sum_{s_i + t_i > 0, \sum (s_i + t_i) = r}\frac{\sum^n_{k = 1}(s_k + t_k)^{-1}}{s_1!t_1! \cdots s_n!t_n!}[X^{s_1}Y^{t_1} \cdots X^{s_n}Y^{t_n}].
\end{equation*}
The terms in the $rhs$ corresponding to this correction can only involve iterated Lie brackets of terms $e_{a'\mu + \gamma}$. These correspond to a sum of terms of the form
\begin{equation*}
\sum_{a_1 < \dots < a_{r-1}, s_1 \dots s_{r-1}, \sum{a_k s_k} = a} P(a_1, \dots, a_{r-1})(\Omega(a_i, 1))^{s_1}\cdots(\Omega(a_{r-1}, 1))^{s_{r-1}}
\end{equation*}
where $s_k \geq 0$, $\sum_k s_k = r$, $s_i \geq 2$ for some $i$, and $P(a_1, \dots, a_{r-1})$ is a homogeneous polynomial of degree $r-1$. 

On the Joyce-Song side the terms corresponding to monomials in $P$ involving precisely $h$ variables coincide with the contribution of the rooted trees with $r$ vertexes labelled by $\gamma$ and precisely $h$ `caps' (in the terminology introduced for $r = 3$). 
\section{From D0-D6 to GW}\label{GW_sect}
In this section we explain how the theory of Gross-Pandharipande-Siebert, in particular the main result from \cite{gps}, the `full commutator formula' Theorem 5.6, applies to the case of D0-D6 states. As we will see from this point of view what links D0-D6 states to GW invariants in the product formula for the McMahon function. We briefly recollect the `full commutator formula' in the form we will need.
\subsection{Orbifold blowups} Let $D \subset S$ be a divisor in a smooth surface and $x \in D$ a smooth point. Smoothness implies that for each $j \geq 1$ there is a unique subscheme of $D$ of length $j$ with reduced scheme $x$. We view this nonreduced scheme as a subscheme $x^j_D \subset S$. For $j \geq 2$ the scheme-theoretic blowup $S_j$ of $S$ along $x^j_{D}$ has a unique singular point of type $A_{j-1}$ lying in the exceptional divisor $E$. For these quotient singularities we can put the structure of a smooth orbifold on $\mathcal{S}_j$ over $S_j$. For example the blowup of $\C^2$ along a length $2$ subscheme $Z$ supported at the origin has an ordinary double point at the point of $E$ corresponding to the direction cut out by $Z$, and so is locally the smooth orbifold $\C^2/\Z_2$. In this case one can check directly that, on the smooth orbifold, $E^2 = -\frac{1}{2}$, and in general one can prove that, on $\mathcal{S}_j$, $E^2 = -\frac{1}{j}$.
\subsection{Graded ordered partitions} A graded ordered partition $P$ is a $d-$tuple $P = (\p^1, \dots, \p^d)$ of ordered partitions such that \emph{every part of $\p^j$ is divisible by $j$}. Its parts are labelled by $p^j_{k}$ for $j = 1,\dots, d$ and $k = 1, \dots, l^j = \len(\p^j)$. We set $\len(P) = (l^1, \dots, l^d)$ and $|P| = \sum_{j, k} p^j_k$.
\subsection{Toric orbifolds} Let $(a,r)$ denote a \emph{primitive} vector. The fan given by $(-1,0), (0,-1)$ and $(a,r)$ defines a toric surface $X_{(a, r)}$, the weighted projective plane $\PP(a, r, 1)$. The faces then correspond to toric divisors $D_1, D_2, D_{out}$. Removing the 3 torus fixed points $[1:0:0], [0:1:0], [0:0:1]$ we obtain a quasi-projective toric orbifold $X^o_{(a,r)}$ with divisors $D^o_1, D^o_2, D^o_{out}$. 

Let now $G = (P_1, P_2)$ be pair of graded ordered partitions $P_1 = (\p^1_1, \dots, \p^{d_1}_1)$, $P_2 = (\p^1_2, \dots, \p^{d_2}_2)$. For $i = 1, 2$ we choose distinct points $x^j_{ik} \in D^o_i$ corresponding to the parts $p^j_{ik}$ of $\p^j_i$. We pick a toric resolution $\widetilde{X}\to X$ whose corresponding divisors $\widetilde{D}_1,\widetilde{D}_2, \widetilde{D}_{out}$ are disjoint and define a smooth orbifold $\widetilde{X}[G]$ over $\widetilde{X}$ as the orbifold blowup of $\widetilde{X}$ at the points $x^j_{ik}$. The underlying singularities become worse as $j$ increases. We denote by $X^o[G] \subset \widetilde{X}[G]$ the preimage of $X^o$. The exceptional divisors are $E^j_{ik}$, and we can define a class $\beta \in H_2(\widetilde{X}, \Z)$ by $\beta\cdot\widetilde{D}_i = |P_i|$ for $i = 1, 2$, $\beta\cdot\widetilde{D}_{out} = \operatorname{ind}(|P_1|,|P_2|)$ (the index of a possibly nonprimitive vector) and $\beta\cdot D = 0$ for all other generators of the Picard group. From $\beta$ we obtain a natural class $\beta_{G}$ on the blowup, i.e. in orbifold cohomology $H_2(\widetilde{X}[G])$, by pulling back and subtracting the weighted exceptional divisors, namely 
\begin{equation*}
\beta_G = \pi^*\beta - \sum_{i = 1,2}\sum^{d_i}_{j = 1}\sum^{l^j_i}_{k = 1} p^j_{ik}[E^{j}_{ik}].
\end{equation*}
\subsection{Moduli spaces of relative stable maps} Gross-Pandharipande-Siebert consider the moduli stack $\overline{M}(\widetilde{X}[G]/\widetilde{D}_{out})$ of genus $0$ stable relative maps in the class $\beta_{G}$ with full tangency of order $\operatorname{gcd}(|P_1|, |P_2|)$ at an unspecified point of the divisor $\widetilde{D}_{out}$, and the open substack $\overline{M}(X^o[G]/D^o_{out})$ given by maps which avoid $\widetilde{X}[G]\setminus X^o[G]$. One of their main technical results (\cite{gps} Proposition 5.5) proves that $\overline{M}(X^o[G]/D^o_{out})$ is proper with a perfect obstruction theory of virtual dimension $0$, so for all $G$ one has well defined GW invariants 
\[N[G] = \int_{[\overline{M}(X^o[G]/D^o_{out})]^{vir}}1 \in \Q.\]
\subsection{Full commutator formula} For $d_1, d_2 \gg 1$ consider the functions
\begin{equation*}
\sigma = \prod^{d_1}_{j = 1}\prod^{l^j_1}_{k = 1}\big(1 + s^j_k x^j\big), \tau = \prod^{d_2}_{j = 1}\prod^{l^j_2}_{k}\big(1+t^{j}_k y^j\big)
\end{equation*}
as elements of the ring of formal power series $\C[[x,y,s^{\bullet}_{\bullet}, t^{\bullet}_{\bullet}]]$ in as many variables as necessary. We define monomials 
\begin{equation*}
s^{P_1} = \prod^{d_1}_{j = 1}\prod^{l^j_1}_{k=1}\big(s^j_k\big)^{\frac{p^j_{1k}}{j}}, t^{P_2} = \prod^{d_2}_{j = 1}\prod^{l^j_2}_{k=1}\big(t^j_k\big)^{\frac{p^j_{2k}}{j}}. 
\end{equation*}
Let $(a, r)\in\Z^2$ be a \emph{primitive} vector. Gross-Pandharipande-Siebert prove a formula for the formal power series $\log f_{(a,r)}$ attached to $(a,r)$ in the ordered product factorisation for the commutator $\tau^{-1}\sigma^{-1}\tau\sigma$, namely
\begin{equation*}
\log f_{m'_{out}} = \sum^{\infty}_{h = 1} \sum_{G = (P_1, P_2)} h N[G]s^{P_1}t^{P_2} x^{h a} y^{h r}
\end{equation*} 
where the sum is over all graded ordered partitions $P_1$ of length $(l^1_1, \dots, l^{d_1}_1)$ and $P_2$ of length $(l^1_2, \dots, l^{d_2}_2)$ such that $(|P_1|, |P_2|) = h\cdot(a, r)$.
\subsection{Application to D0-D6 states} In our case we have
\begin{equation*}
\widetilde{\sigma} = \prod_{n \geq 1}(1-(-u)^n x^n)^{n \chi}, \tau = (1-u y).
\end{equation*}
Clearly then $d_2 = 1, l^2_1 = 1$ and $t^1_1 = -u$. On the other hand we can truncate $\sigma$ to a fixed $d_1\gg 1$ and write 
\begin{equation*}
\sigma = \prod^{d_1}_{j = 1} \prod^{j\chi}_{k = 1} (1-(-u)^j x^j),
\end{equation*} 
which corresponds to the choices
\begin{equation*}
l^1_j = j \chi,\,\,\,s^j_{k} = (-1)^j u^j \text{ for } k = 1, \dots, j \chi.
\end{equation*}
Now fix a primitive vector $(a, r)$. The admissible ordered partitions $(P_1, P_2)$ actually have the form $(P_1, h r)$ for some $h \geq 1$, so $t^{P_2} = (-1)^{k r} u^ {k r}$. On the other hand $P_1 = (\p^1_1, \p^2_1, \dots, \p^{d_1})$ is a $d_1-$tuple of ordered partitions, with $\len(\p^j_1) = j\chi$ and $|P_1| = h a$, and where each part of $\p^j_1$ is divisible by $j$. It follows that 
\begin{equation*}
s^{P_1} = \prod^{d_1}_{j = 1}\prod^{j \chi}_{k = 1}((-1)^j u^j)^{\p^{j}_{1k}} = (-1)^{h a} u^{h a}.
\end{equation*} 
By the full commutator formula then
\begin{equation*}
\log f_{(a, r)} = \sum^{\infty}_{h = 1} \sum_{P_{\chi}} h N[P_{\chi}](-1)^{h(a + r)}(u x)^{h a}(u y)^{h r}
\end{equation*}
where the sum is over all graded ordered partitions $P_{\chi}$ with length vector 
\begin{equation*}
\len P_{\chi} = (\chi, 2\chi, \dots, d_1\chi)
\end{equation*}
and $|P_{\chi}| = h a$. We have obtained the required D0-D6/GW duality in the ring $\C[[x, y]]$
\begin{equation}
\prod_{h \geq 1} \exp\left(\sum_{P_{\chi}} h N[P_{\chi}](-1)^{h( a +  r)}(x)^{h a}(y)^{h r}\right) = \prod_{h \geq 1}(1 - (-1)^{h^2 a r}(x)^{h a} (y)^{h r})^{\Omega(h a, h r)}.
\end{equation} 
The two sets of invariants are completely determined through each other.
\section{Appendix}
In this appendix we will illustrate the Joyce-Song D0-D6 invariants by computing $\bar\dt(2,2)$ using Behrend functions methods. This is possible by a formula of Joyce-Song (see \cite{joy} equation (16)) which connects $\bar\dt_{Gies}$ invariants to other invariants $\pii^m$ (depending on a parameter $m \gg 1$) called \emph{pair invariants}. These are obtained rigidifying with a section so they become virtual counts and are given by weighted Euler characteristics (we will define them in a moment). Here $\bar\dt_{Gies}$ denotes the invariants counting Gieseker semistable objects in the category $\mathcal{A}$. In turn it is possible to recover $\bar\dt$ from $\bar\dt_{Gies}$.\\

By definition a Joyce-Song stable pair is given by a nonzero section 
\[s\!: \O_X(-m) \to F\]
where $F$ is \emph{Gieseker semistable}, $s$ does not factor through a semi-stabilising subsheaf and $m$ is greater than the Castelnuovo-Mumford regularity of $F$, all this modulo the natural gauge equivalence relation. Stable pairs have a fine moduli scheme which we denote by $\mathcal{M}^m_{stp}(\beta)$. This Hilbert scheme of stable pairs has a symmetric obstruction theory and so gives virtual counts (fixing a class $\beta$)
\[\pii^m(\beta) = \int_{[\mathcal{M}^m_{stp}(\beta)]^{vir}}1 = \chi(\mathcal{M}^m_{stp}(\beta), \nu_{\mathcal{M}^m_{stp}(\beta)}).\] 

We also write $\mathcal{M}_{Gies}$ for the Artin stack of Gieseker semistable coherent sheaves with the same Chern character. There is a natural representable morphism 
\[\pi\!: \mathcal{M}^m_{stp} \to \mathcal{M}_{Gies}\]
given by forgetting the morphism $s$.\\ 

The Joyce-Song formula \cite{joy} equation (16) holds for any effective class $\alpha \in K^+(X)$ and $m \gg 1$, giving
\begin{equation}\label{fromPairs}
\pii^m(\alpha) = \sum_{n \geq 1, \{\alpha_i\}^n_{i = 1} \subset (K^+(X))^n, \sum^n_{i = 1} \alpha_i = \alpha, \forall i P_{\alpha_i} = P_{\alpha}} \frac{(-1)^n}{n!} F(\alpha_1, \dots, \alpha_n)
\end{equation}
where
\begin{equation}
F(\alpha_1, \dots, \alpha_n) = \prod^n_{j = 1} (-1)^{\bra\O_X(-m)-\sum^{j-1}_{i = 1}\alpha_i, \alpha_j\ket}\bra\O_X(-m)-\sum^{j-1}_{i = 1}\alpha_i, \alpha_j\ket\bar\dt_{Gies}(\alpha_j).
\end{equation}
while $P_{\bullet}$ denotes the Hilbert polynomial and $\bra\bullet, \bullet\ket$ the Mukai pairing.\\

Let now $F$ be a torsion free sheaf which is isomorphic to $\O_X^{\oplus 2}$ away from a zero dimensional subscheme of $X$. By \cite{oko} Lemma 1.1.8 we can extend this isomorphism to an inclusion $F \subset \O^{\oplus 2}_X$. Intersecting with a generic copy of $\O_X \subset \O_X \oplus \O_X$ gives $F \cap \O_X \cong \I_Z$ for some zero dimensional subscheme $Z \subset X$. So we see that $F$ is an extension 
\[0 \to \I_Z \to F \to \I_W \to 0\]
for some other zero dimensional $W \subset X$. 

If moreover $F$ is Gieseker semistable with $\operatorname{ch}(F) = (2, 0 ,0 , -2)$ we see that it must be either a \emph{nonsplit} extension
\begin{equation}\label{len2}
0 \to \I_Z \to F \to \O_X \to 0
\end{equation}
with $\operatorname{len}(Z) = 2$, or 
\begin{equation}\label{points}
0 \to \I_p \to F \to \I_q \to 0
\end{equation}
for closed points $p, q \in X$. One can show that these sheaves are Gieseker semistable, and there is no intersection bewteen the sets of sheaves appearing in \ref{len2}, \ref{points}.\\

For any $0$-dimensional subscheme $Z$ we have the exact sequence
\[0 = H^0(\I_Z) \to H^0(\O_X) \to H^0(\O_Z) \to H^1(\I_Z) \to H^1(\O_X) = 0\]
so the sheaves $F$ of the form \eqref{len2} are parametrised by 
\[\PP(\Ext^1(\O_X, \I_Z)) = \PP(H^1(\I_Z)) \cong \Spec(\C).\]
Proceding a bit further with the same exact sequence we find
\[0 = H^1(\O_Z) \to H^2(\I_Z) \to H^2(\O_X) = H^1(\O_X)^{\vee} = 0.\]
In particular $\Ext^1(\I_q, \I_p) = 0$ for $p \neq q$. To see this take RHom$(\I_q, \bullet)$ of the sequence $0 \to \I_p \to \O_X \to \O_p \to 0$ to find
\begin{align*}
0 = \Hom(\I_q, \I_p) \to \Hom(\I_q, \O_X) &\to \Hom(\I_q, \O_p)\\
&\to \Ext^1(\I_q, \I_p) \to \Ext^1(\I_q, \O_X) \cong H^2(\I_q) = 0, 
\end{align*}
and note that since $p \neq q$ the map $\Hom(\I_q, \O_X) \to \Hom(\I_q, \O_p)$ is onto. So for $p \neq q$ we only have the split extension $F = \I_p \oplus \I_q$. 

On the other hand $\PP(\Ext^1(\I_p, \I_p)) \cong \PP(T_p\Hilb^1(X))\cong \PP(T_p X)$.\\

To compute $\pii^m(2,2)$ we stratify $\mathcal{M}^m_{stp}$ according to the image of $\pi$; this gives strata over which Behrend's constructible function is constant. Let us consider 
\[\mathcal{M}_1 \subset \mathcal{M}^m_{stp}\]
given by pairs projecting to sheaves $\I_p \oplus \I_q$ with $p \neq q$. The fibres of $\pi|_{\mathcal{M}_1}$ are all isomorphic to $\PP(H^0(\I_p(m)))\times\PP(H^0(\I_q(m)))$ with \emph{even} dimension $2(P_X(n)-2)$ and Euler characteristic $(P_X(m) - 1)^2$. Since the fibres are smooth with even dimension we see 
\[\nu_{\mathcal{M}_1} = \pi^*\nu_{\mathcal{M}_{Gies}}\]
(the sign would be negative for odd dimensional fibres, see e.g. \cite{joy} Section 1.2 for this and other general properties of Behrend functions) and by \cite{joy} equation (10)
\begin{align*}
\nu_{\mathcal{M}_{Gies}}(\I_p \oplus \I_q) &= (-1)^{\bra\I_p, \I_q\ket} \nu_{\mathcal{M}_{Gies}}(\I_p)\cdot\nu_{\mathcal{M}_{Gies}}(\I_q)\\
&= \nu_{\mathcal{M}_{Gies}}(\I_p)\cdot\nu_{\mathcal{M}_{Gies}}(\I_q)\\
&= (-1)(-1) = 1.
\end{align*}
Integrating over unordered pairs $\{p, q\}$ is half of integrating away from the diagonal in $X \times X$, so we find that the $\mathcal{M}_1$ contribution to $\pii^m(2,2)$ is
\begin{align*}
\int_{\mathcal{M}_1}\nu_{\mathcal{M}^m_{stp}}d\chi = \int_{\mathcal{M}_1}\pi^*\nu_{\mathcal{M}_{Gies}}d\chi &= \frac{1}{2}(P_X(m) - 1)^2 \int_{(X\times X)\setminus \Delta} d\chi\\
&= \frac{1}{2}(P_X(m) - 1)^2 \chi((X \times X)\setminus\Delta)\\
&= \frac{1}{2}(P_X(m) - 1)^2(\chi^2 - \chi). 
\end{align*}

Similarly for pairs $\mathcal{M}_2$ projecting to $\I_p \oplus \I_p$ the fibre of $\pi$ is $\Gr(2, H^0(\I_p(m)))$ with \emph{even} dimension $2(P_X(m)-3)$ and Euler characteristic 
\[\binom{P_X(m)-1}{2} = \frac{1}{2}(P_X(m)-1)^2 - \frac{1}{2}(P_X(m)-1)\]
while $\nu_{\mathcal{M}_{Gies}}(\I_p \oplus \I_p) = \nu_{\mathcal{M}_{Gies}}(\I_p)^2 = 1$, yielding
\[\int_{\mathcal{M}_2}\nu_{\mathcal{M}^m_{stp}}d\chi = \int_{\mathcal{M}_2} \pi^* \nu_{\mathcal{M}_{Gies}}d\chi = \frac{1}{2}(P_X(m)-1)^2 \chi - \frac{1}{2}(P_X(m)-1)\chi.\]

Consider next a sheaf $F$ given by a nonsplit extension
\[0 \to \I_p \to F \to \I_p \to 0.\]
Its deformations correspond to deformations of $p \in X$ plus deformations of $[F]$ in $\PP(\Ext^1(\I_p, \I_p)) \cong \PP^2$, so 
\[\nu_{\mathcal{M}_{Gies}}(F) = (-1)^2(-1)^3 = -1.\]
The set of all extensions above forms the projective bundle $\PP(TX)\to X$. Let $\mathcal{M}_3$ denote the locus of pairs projecting to extenstions as above. The fibre $\pi^{-1}(F)$ is given by $H^0(\I_p(m)) \times \PP(H^0(\I_p(m)))$ with \emph{odd} dimension $2P_X(n) - 3$ and Euler characteristic $P_X(n)-1$, so we find
\[\int_{\mathcal{M}_3}\nu_{\mathcal{M}^m_{stp}}d\chi = -\int_{\mathcal{M}_3} \pi^* \nu_{\mathcal{M}_{Gies}}d\chi = 3 (P_X(m)-1) \chi.\]

The last kind of pairs $\mathcal{M}_4$ project to a nonsplit extension 
\[0 \to \I_Z \to F \to \O_X \to 0\]
for a length 2 subscheme $Z \subset X$. The sheaf $F$ is stable, and admissible sections are $\PP(H^0(F))$ with \emph{odd} dimension $2P_X(n)-3$ and Euler characteristic $2(P_X(n)-1)$. As we have seen $\ext^1(\O_X, \I_Z) = 1$ and so these sheaves are parametrised by the Hilbert scheme $\Hilb^2(X)$. It remains to compute $\nu_{\mathcal{M}_{Gies}}(F)$. A nice way to do this is by the fundamental Joyce-Song duality \cite{joy} (10), (11) Section 1.3, which gives
\begin{align*}
\int_{[E]\in\PP(\Ext^1(\O_X, \I_Z))}\nu_{\mathcal{M}_{Gies}}(E) d\chi &= \int_{[E]\in\PP(\Ext^1(\I_Z,\O_X))}\nu_{\mathcal{M}_{Gies}}(E) d\chi\\
&+ (\ext^1(O_X, \I_Z) - \ext^1(\I_Z, \O_X))\nu_{\mathcal{M}_{Gies}}(\I_X \oplus\O_X)\\
&= -\nu_{\mathcal{M}_{Gies}}(\I_X \oplus\O_X)\\
&= -(-1)^{\bra\I_X, \O_X\ket}\nu_{\mathcal{M}_{Gies}}(\I_Z)\cdot\nu_{\mathcal{M}_{Gies}}(\O_X)\\
&= -\nu_{\mathcal{M}_{Gies}}(\I_Z),
\end{align*}
using $\ext^1(\I_Z, \O_X) = 0$ and the fact that $\O_X$ is rigid by assumption. Since the left hand side is just $\nu_{\mathcal{M}_{Gies}}(F)$ we see this equals $-\nu_{\mathcal{M}_{Gies}}(\I_Z)$. Integrating over all pairs in $\mathcal{M}_4$ gives
\[\int_{\mathcal{M}_4}\nu_{\mathcal{M}^m_{stp}}d\chi = -\int_{\mathcal{M}_3}\pi^*\nu_{\mathcal{M}_{Gies}}d\chi = (P_X(n)-1)(\chi^2 + 5\chi).\]
 
Putting together these computations we see that $\pii(2,2) = \sum^4_{i = 1}\int_{\mathcal{M}_i}\nu_{\mathcal{M}^m_{stp}}d\chi$ is a polynomial in the `variable' $P_X(n)-1$, 
\begin{equation}
\pii(2,2) = \frac{1}{2}\chi^2(P_X(m)-1)^2 + \frac{\chi^2 + 15\chi}{2}(P_X(m)-1). 
\end{equation}

Let us now extract the $\bar\dt$ invariants from $\pii(2,2)$. According to \eqref{fromPairs} the only contributions to $\pii(2,2)$ are 
\begin{align*}
-F_1 &= -2(P_X(n)-1)\bar\dt_{Gies}(2,2),\\
\frac{1}{2}F_2 &= \frac{1}{2}(P_X(n)-1)^2 (\bar\dt_{Gies}(1,1))^2\\
&= \frac{1}{2}(P_X(n)-1)^2 \chi^2.
\end{align*}
Comparing with $\pii(2,2)$ gives
\[\bar\dt_{Gies}(2,2) = -\frac{5}{4}\chi - \frac{\chi^2 + 5\chi}{2}.\]
The wall-crossing to $\bar\dt$ is especially simple in this case,  
\[\bar\dt(2,2) = \bar\dt_{Gies}(2,2) + \bar\dt_{Gies}(2,1).\]
Therefore we find
\[\bar\dt(2,2) = -\frac{5}{4}\chi,\,\,\,\Omega(2,2)=-\chi.\] 

\end{document}